\documentclass[12pt,english]{article}
\usepackage[latin9]{inputenc}
\usepackage{geometry}
\usepackage{babel}
\usepackage{amsmath}
\usepackage{amsthm}
\usepackage{amssymb}
\usepackage{graphicx}
\usepackage[unicode=true,pdfusetitle,
 bookmarks=true,bookmarksnumbered=false,bookmarksopen=false,
 breaklinks=false,pdfborder={0 0 1},backref=false,colorlinks=false]
 {hyperref}

\usepackage{tikz}
\usetikzlibrary{calc,arrows.meta,positioning}
\tikzset{
    every node/.style={font=\sffamily\small},
    main node/.style={thick,circle,draw,font=\sffamily,  inner sep = 3pt, minimum size = 10pt}
}

\makeatletter
\theoremstyle{plain}
\newtheorem{thm}{\protect\theoremname}
\theoremstyle{plain}
\newtheorem{lem}[thm]{\protect\lemmaname}

\date{}

\makeatother

\providecommand{\lemmaname}{Lemma}
\providecommand{\theoremname}{Theorem}

\begin{document}
\title{The Toucher-Isolator Game on Trees}
\author{Eero R\"{a}ty\thanks{Centre for Mathematical Sciences, Wilberforce Road, Cambridge CB3
0WB, UK, epjr2@cam.ac.uk}}
\maketitle
\begin{abstract}
Consider the following Maker-Breaker type game played by Toucher and
Isolator on the edges of a graph $G$ with first move given to Toucher.
The aim of Isolator is to maximise the number of vertices which are
not incident to any edges claimed by Toucher, and the aim of Toucher
is  to minimise this number. Let $u\left(G\right)$ be the number
of isolated vertices when both players play optimally. Dowden, Kang,
Mikala\v{c}ki and Stojakovi\'{c} proved that $\left\lceil \frac{n+2}{8}\right\rceil \le u\left(T\right)\leq\left\lfloor \frac{n-1}{2}\right\rfloor $,
where $T$ is a tree with $n$ vertices. The author also proved that $u\left(P_{n}\right)=\left\lfloor \frac{n+3}{5}\right\rfloor$
for all $n\geq3$, where $P_{n}$ is a path with $n$ vertices

The aim of this paper is to improve the lower bound to $u\left(T\right)\geq\left\lfloor \frac{n+3}{5}\right\rfloor$,
which is sharp. Our result may be viewed as saying that paths are the 'best' for Isolator among trees with a given number of vertices. 
\end{abstract}

\section{Introduction}

The following game, called 'Toucher-Isolator' game on a graph $G$,
was introduced by Dowden, Kang, Mikala\v{c}ki and Stojakovi\'{c} \cite{key-7}.
The two players, Toucher and Isolator claim edges of $G$ alternately
with Toucher having the first move. We say that a vertex is \emph{isolated} if it is not 
incident to any of the edges claimed by Toucher at the end of the game. The aim of 
Toucher is to minimise the number of isolated vertices and the aim of Isolator is to 
maximise the number of isolated vertices. We set $u\left(G\right)$ to be the number
of isolated vertices at the end of the game when both players play optimally. Hence this
is a 'quantitative' Maker-Breaker type game. 

Dowden, Kang, Mikala\v{c}ki and Stojakovi\'{c} gave bounds \cite{key-7} for
the size of $u\left(G\right)$ for general graphs, and they studied
some particular examples. In particular, they focused on the case
when $T$ is a tree and gave bounds for $u\left(T\right)$ in terms
of the degree sequence of $T$. They also proved that if $T$ is a
tree with $n$ vertices, then 
\begin{equation}
\frac{n+2}{8}\leq u\left(T\right)\leq\frac{n-1}{2}.\label{eq:1}
\end{equation}
If $T$ is a star with $n\geq3$ vertices, it is easy to verify that
$u\left(T\right)=\left\lfloor \frac{n-1}{2}\right\rfloor $ regardless
of how Toucher and Isolator play the edges. Hence the upper bound
in (\ref{eq:1}) is tight. 

Let $P_{n}$ be the path with $n$ vertices. For paths, Dowden, Kang,
Mikala\v{c}ki and Stojakovi\'{c} improved their general bound (\ref{eq:1})
to

\[
\frac{3}{16}\left(n-2\right)\leq u\left(P_{n}\right)\leq\frac{n+1}{4},
\]
and they suggested that $u\left(P_{n}\right)$ could asymptotically
grow as $\frac{n}{5}$. The author \cite{key-9} verified this by
proving an exact result which states that $u\left(P_{n}\right)=\left\lfloor \frac{n+3}{5}\right\rfloor $
for all $n\geq3$. 

The aim of this paper is to improve the lower bound in (\ref{eq:1}),
and in particular we prove that paths are the 'best' for Isolator
among the trees with $n$ vertices. 
\begin{thm} \label{thm:1}
Let $n\geq3$ and let $T$ be a tree with $n$
vertices. Then $u\left(T\right)\geq\left\lfloor \frac{n+3}{5}\right\rfloor $. 
\end{thm}

Indeed, since $u\left(P_{n}\right)=\left\lfloor \frac{n+3}{5}\right\rfloor $
for all $n\geq3$, it follows that $u\left(T\right)\geq u\left(P_{n}\right)$
for any $T$ with $n$ vertices. In general, it is easy to see that
the lower bound is not uniquely attained by a path. Indeed, let $S_{n}$
be a tree consisting of a path with $n-1$ vertices together with
a leaf joined to the second vertex on the path. It is easy to prove
by following a similar argument that was used for paths in \cite{key-9}
that we also have $u\left(S_{n}\right)=u\left(P_{n}\right)$.

The proof of Theorem \ref{thm:1} follows an approach that is similar to 
the proof of lower bound of $u\left(P_{n}\right)$ in \cite{key-9}. For general background 
on Maker-Breaker type games, see Beck \cite{key-2}. There are many other 
papers dealing with achievement games on graphs, see 
e.g.\ \cite{key-1,key-3,key-8}. 

We now outline the proof of Theorem 1 as it is rather lengthy. As in \cite{key-9}, 
the game naturally splits into two phases. 
At the early stages of the game it seems natural for
Isolator to claim an edge whose endpoint is a leaf, as claiming such
an edge instantly increases the score by one. Note that after claiming
such an edge, the isolated leaf in a sense becomes 'useless' for the
rest of the game, and hence it may be deleted (together with the edge
that was occupied). Note that during the process the other endpoint
of the edge may become a leaf. We say that the game is in the \textit{first
phase} as long as such an edge always exists, and once such edge no
longer exists the game moves to the \textit{second phase}. In particular,
the game will not return to the first phase even though such an edge would
become available later. 

Suppose that at some point Isolator has no such move available. Let
$T_{1}$ be the tree obtained as a result of the process, and let $C_{1}$ be
the set of edges claimed by Toucher. Thus for every leaf $v\in T_{1}$
there exists an edge $e\in C_{1}$ whose endpoint $v$ is. This is
quite similar to the delayed game introduced in \cite{key-9}, in
which Toucher is allowed to claim a certain number of edges at the
start of the game and in which isolating endpoints, which naturally
corresponds to leaves in our general case, do not increase the score. 

However, before we can define a delayed game  that is good enough
for our purposes, the structure of $T_{1}$ and $C_{1}$ may need
to be modified. Let $v$ be a leaf in $T_{1}$ whose unique neighbour
is $w$. Similarly to leaves that are already isolated, the leaves
that are already touched are quite useless for the rest of the game,
so it would be tempting to just delete them. However, during the 
process we must make sure that we keep in mind that $w$ is already
touched, even though the edge $vw$ is deleted during the process. Thus
it will be convenient to declare a set $X_{1}\subseteq V\left(T_{1}\right)$
of 'additional' touched vertices at the start of the delayed
game. 

It turns out that such a simple deletion is fine whenever $d_{T_{1}}\left(w\right)=2$,
but when $d_{T_{1}}\left(w\right)\geq3$ we need to modify the structure
of $T_{1}$ in a slightly different way. In this case the modification
is based on the observation that for a given edge $e=vw$, if $w$
is a touched vertex then $e$ can be replaced with an edge $vw'$
for any touched vertex $w'$ without changing the game too much. By
using this observation, we may restrict ourselves to those delayed
games where $X_{1}$ is exactly the set of all leaves in $T_{1}$. 

The plan of the paper is following. In Section 2 we define the notion that allows us 
to reduce the tree without changing the game too much, and we deal
with the first phase of the case in which Isolator is claiming only
leaves. In Section 3 we move on to analysing the specific delayed
version of the game where leaves are not counted for isolated vertices,
and Toucher is allowed to claim a certain number of edges at the start
of the game. We then use these delayed games to deal with the second
phase of the game. 

\section{The First phase of the game}

We start by introducing the notion of the delayed game. Let $T$ be a
tree, let $C$ and $D$ be disjoint subsets of the edges of $T$,
and let $X$ be a subset of the vertices of $T$. Define the \textit{delayed
game} $F\left(T,C,D,X,s\right)$ to be the Toucher-Isolator game played
on the edges of $T$, with the edges in $C$ and $D$ given to Toucher
and Isolator respectively at the start of the game, and with both
players claiming the edges in alternating turns with the first move given
to the player specified by the parameter $s\in\left\{ \text{i},\text{t}\right\} $.
Define the \textit{score} of this game to be the number of isolated
vertices in $V\left(T\right)\setminus X$ at the end of the game,
and denote the score by $\alpha\left(T,C,D,X,s\right)$. 

For our purposes, we mostly focus on certain sub-classes of these
games, and hence some of the parameters can be omitted as they will
be clear from the context. First of all, we use $F\left(T\right)$
to denote the ordinary Toucher-Isolator game on $T$, that is $F\left(T,\emptyset,\emptyset,\emptyset,\text{t}\right)$,
and similarly we use $\alpha\left(T\right)$ to denote the score of
$F\left(T\right)$. However, apart from this special case, it is more
convenient to choose Isolator to start the delayed version of the
game, and hence $s$ should be taken to be Isolator if it is omitted
from the notation, with $F\left(T\right)$ being an exception. Similarly
$C$ and $D$ should be taken to be empty sets if they are omitted
from the notation. We often either have $X=\emptyset$ or $X=L$,
where $L$ denotes the set of leaves in $T$. Hence we use $F\left(T,C,X\right)=F\left(T,C,\emptyset,X,\text{i}\right)$,
$F\left(T,C\right)=F\left(T,C,\emptyset,\emptyset,\text{i}\right)$
and $F\left(T,C,L\right)=F\left(T,C,\emptyset,L,\text{i}\right)$
to simplify our notation. 

Since some of the results used in the proof of Theorem \ref{thm:1}
are proved by induction, it is convenient to introduce a suitable
reduction operation that allows us to reduce the tree without increasing
the score of the game. Our reduction operator is defined for the games
of the form $F\left(T,C,D,X,s\right)$, and in general $s$ is taken
to be Isolator. 

First we need some notation. As usual, let $E$ and $V$ denote the
set of edges and vertices in $T$ respectively, and let $C$, $D$
and $X$ be defined as before. Let $\hat{E}=E\setminus\left(C\cup D\right)$
be the set of edges that are not given to Toucher or Isolator at the
start of the game, let $I$ be the set of vertices in $V\setminus X$
that are isolated by the edges in $D$, and let $O$ be the set of
vertices in $V\setminus X$ that are touched by an edge in $C$. The
vertices in $O$ are called \textit{occupied} and the vertices in
$O\cup X$ are called \textit{touched}. Finally we set $U=V\setminus\left(I\cup X\cup O\right)$,
and the vertices in $U$ are called \textit{unoccupied}. Note that
the set of vertices that are not yet isolated nor touched is $U$,
and hence $U$ is exactly the set of those vertices that could be
still isolated. 

The definition of the reduction operation is quite tedious, but the
ideas behind it are fairly simple, and we start by outlining these
ideas. Suppose that $v_{1}$ and $v_{2}$ are two touched vertices
and let $e$ be an edge of the form $uv_{1}$ that is not in $C$.
Let $T_{1}$ be the graph obtained by replacing the edge $uv_{1}$
with $uv_{2}$ in $T$, and suppose that $T_{1}$ is also a tree.
This operation changes the structure of $T$, but does not affect
the game at all. First of all, note that the process might only affect
the vertices $u$, $v_{1}$ and $v_{2}$. Note that in both $T$ and
$T_{1}$ the vertices $v_{1}$ and $v_{2}$ are already touched and hence they are not 
affected during the process. Also only one edge with $u$ as its endpoint is affected during the process
and in both $T$ and $T_{1}$ the other endpoint is touched. Note that it does
not matter which particular vertex the other endpoint is, as long
as in both cases the other endpoint is touched. Hence in fact the
game is not affected at any vertex during the process. One can also
perform similar operations to leaves that are endpoint of an edge
in $C$ and whose neighbour has degree at least $3$. 

Recall that $\hat{E}=E\setminus\left(C\cup D\right)$. Let $e\in\hat{E}$,
and note that hence neither of the endpoints of $e$ is in $I$. Define
the \textit{endpoint pattern} of $e$ to be $P\left(e\right)\in\left\{ 1,2,3\right\} $,
where $P\left(e\right)=1$ if both endpoints of $e$ are unoccupied,
$P\left(e\right)=2$ if one of the endpoints is occupied and the other
is touched and $P\left(e\right)=3$ if both endpoints are touched.
Let $T_{1}$ and $T_{2}$ be trees with appropriate sets $C_{i},\,D_{i}$
and $X_{i}$. We say that a function $f:\hat{E}\rightarrow\hat{E}_{2}$
\textit{preserves the type of the endpoints} if for every $e\in\hat{E}_{1}$,
$e$ and $f\left(e\right)$ have the same endpoint pattern. Finally
for a vertex $v\in V\left(T\right)$ define ${\cal E}\left(v\right)$
to be the collection of edges whose endpoint $v$ is. 

We say that $F\left(T_{1},C_{1},D_{1},X_{1},s\right)$ is a \textit{reduction}
of $F\left(T_{2},C_{2},D_{2},X_{2},s\right)$ if $D_{1}=\emptyset$
and if there exist injections $f_{E}:\hat{E}_{1}\rightarrow\hat{E}_{2}$
and $f_{V}:U_{1}\rightarrow U_{2}$ so that $f_{E}$ preserves the
type of the endpoints and we have ${\cal E}\left(f_{V}\left(v\right)\right)=f_{E}\left({\cal E}\left(v\right)\right)$
for all $v\in U_{1}$, where $f_{E}\left(A\right)=\bigcup_{e\in A}f_{E}\left(a\right)$
for $A\subseteq\hat{E}_{1}$. The first condition is intuitively clear
and the second condition implies that the neighbourhood of an unoccupied
vertex is preserved, which is crucial as we want isolating new vertices
to be a similar process in both $T_{1}$ and $T_{2}$. For convenience,
we just say that $T_{1}$ is a reduction of $T_{2}$ if $F\left(T_{1},C_{1},D_{1},X_{1},s\right)$
is a reduction of $F\left(T_{2},C_{2},D_{2},X_{2},s\right)$, as the
other parameters are clear from the context. 

In all of our applications, $T_{1}$ is obtained by deleting some
vertices from $T$ or by changing endpoints of several edges. If only
deletion of vertices is used in the process, we usually take $f_{E}$
and $f_{V}$ to be the identity maps. For convenience, if $f_{E}$
and $f_{V}$ are taken to be the identity maps we simply say that
$T_{1}$ is a reduction of $T$ (without explicitly specifying that
the maps are taken to be identity maps). If endpoints of some edges
are changed, we often still take $f_{V}$ to be the identity map and
we take $f_{E}\left(e\right)=e$ for most of the edges, apart from
several exceptions involving the edges whose endpoints were changed.
In such a case we specify the map $f_{E}$ only on these exceptional
edges, and for any unspecified $v\in U_{1}$ and $e\in\hat{E}_{1}$
one should take $f_{V}\left(v\right)=v$ and $f_{E}\left(e\right)=e$.

Our first aim is to prove that such reduction operation cannot increase
the score, when the effect of those vertices that are isolated already
is taken into account. This essentially follows by copying the strategy
on $T_{1}$ to a strategy on $T_{2}$ by using the function $f_{E}$. 
\begin{lem}
\label{lem:Lemma 2}Let $T_{1}$ and $T_{2}$ be trees with appropriate
sets $C_{i},\,D_{i}$ and $X_{i}$ with $D_{1}=\emptyset$, and suppose
that $T_{1}$ is a reduction of $T_{2}$. Let $I_{2}$ be the set
of isolated vertices in $T_{2}$. Then $\alpha\left(T_{2},C_{2},D_{2},X_{2},s\right)\geq\left|I_{2}\right|+\alpha\left(T_{1},C_{1},D_{1},X_{1},s\right)$.
\end{lem}

\begin{proof}
Let $S_{1}$ be a strategy on $T_{1}$ which guarantees that Isolator
can isolate at least $\alpha\left(T_{1},C_{1},D_{1},X_{1},s\right)$
vertices. Consider the strategy $S_{2}$ on $T_{2}$ obtained as follows.
If on her move Toucher claims an edge $e\in\hat{E}_{2}$ for which
$e$ is in the image of $f_{E}$, then the edge $f_{E}^{-1}\left(e\right)$
is assigned to Toucher on $T_{1}$. If she claims an edge $e\in\hat{E}_{2}$
that is not in the image of $f_{E}$, then an arbitrary edge is assigned
to Toucher on $T_{1}$. On a given turn, if Isolator claims
an edge $g\in\hat{E}_{1}$ according to the strategy $S_{1}$, then
on $T_{2}$ she claims the edge $f_{E}\left(g\right)$. Once all the
edges on $T_{1}$ are occupied, Isolator always plays an arbitrary
edge on $T_{2}$ on her move. 

By following this strategy, at the end of the game Isolator has isolated
$\alpha\left(T_{1},C_{1},D_{1},X_{1},s\right)$ vertices on $T_{1}$.
Since for all $v\in U_{1}$ we have $f_{E}\left({\cal E}\left(v\right)\right)={\cal E}\left(f_{V}\left(v\right)\right)$,
it follows that for each isolated vertex $v\in U_{1}$ the appropriate
vertex $f_{V}\left(v\right)$ is also isolated, and all of these vertices
are distinct as $f_{V}$ is an injection. In addition, all the vertices
in $I_{2}$ are isolated as well by definition, and note that $I_{2}\cap U_{2}=\emptyset$.
Hence it follows that $\alpha\left(T_{2},C_{2},D_{2},X_{2},s\right)\geq\left|I_{2}\right|+\alpha\left(T_{1},C_{1},D_{1},X_{1},s\right)$,
as required. 
\end{proof}
Recall that we start the game in the first phase, and after
a given move of Toucher the game remains in the first phase if there
exists an unoccupied vertex $v$ for which ${\cal E}\left(v\right)$
contains exactly one edge that is not already claimed by Isolator.
Otherwise the game moves to the second phase, and note that
this transition always occurs after Toucher's move. In particular,
the game is in the first phase as long as Toucher can increase her
score on every move by claiming a suitable edge - and it turns out
that choosing an arbitrary edge among all such edges will work for
Isolator. 

Let $C$ and $D$ be the set of edges occupied by Toucher and Isolator
when the game moves from the first phase to the second phase. Recall that
for a tree $T$ we write $L$ for the set of leaves in $T$. Our first
aim is to show that there exists a reduction $T'$ of $T$ with $D'=\emptyset$,
$X'=L'$ and for which $\left|T'\right|-3\left|C'\right|-3\left|L'\right|$
is not too small. This is done in Lemma \ref{lem:Lemma 3}. Note that
the game on $T'$ is exactly the delayed game $F\left(T',C',L'\right)$,
as $X'=L'$, $D'=\emptyset$ and since Isolator has the first move
in the second phase. Thus in order to analyse the second phase, we
need a lower bound for $\alpha\left(T,C,L\right)$. In Lemma \ref{lem:Lemma 4}
we prove a lower bound for $\alpha\left(T,C,L\right)$ that depends
on $\left|T'\right|-3\left|C'\right|-3\left|L'\right|$.
\begin{lem}
\label{lem:Lemma 3}Let $T$ be a tree with $n\geq3$ vertices. Suppose
that Isolator has the move, the game is in the first phase and let
$Y$ be the set of those edges $v$ in $E$ that are not played yet
for which Isolator can isolate a new vertex by claiming $v$ on this
move. 

Suppose that on each of her move Isolator claims an arbitrarily chosen
edge from $Y$. Let $r$ be the number of edges Isolator claims during
the first phase, and let $C$ and $D$ be the set of edges claimed
by Toucher and Isolator at the end of the first phase. Then there
exists a reduction $T'$ of the game $F\left(T,C,D\right)$ with $X'=L'$,
$D'=\emptyset$ and $\left|T'\right|-3\left|L'\right|-3\left|C'\right|\geq\left|T\right|-5r-4$. 
\end{lem}

\begin{proof}
Let $C=\left\{ e_{1},\dots,e_{r+1}\right\} $ and $D=\left\{ f_{1},\dots,f_{r}\right\} $
be the set of edges claimed by Toucher and Isolator respectively at
the end of the first phase, ordered in a way that $f_{i}$ is claimed
before $f_{j}$ for $i<j$. Let $v_{i}$ be a vertex isolated by claiming
the edge $f_{i}$. 

We start by verifying that $T\setminus\left\{ v_{1},\dots,v_{i}\right\} $
is a tree for all $i$ and that claiming $f_{i}$ cannot isolate
both of its endpoints. Indeed, note that $v_{1}$ must be leaf, and
since $n\geq3$ it follows that no two leaves can be neighbours.
Hence the claim is true when $i=1$. If the claim is true for all
$1\leq j\leq i$ for some $i$, it follows that $T\setminus\left\{ v_{1},\dots,v_{i}\right\} $
is a tree which does not contain any isolated vertices. Note that
it also does not contain any edge claimed by Isolator on her first
$i$ moves, as for every edge $f_{1},\dots,f_{i}$ at least one of
the endpoints is deleted during the process. Thus $v_{i+1}$ must
be a leaf in $T\setminus\left\{ v_{1},\dots,v_{i}\right\} $, and
note that the unique neighbour of $v_{i+1}$ cannot be a leaf. Indeed,
this follows from the fact that $T\setminus\left\{ v_{1},\dots,v_{i}\right\} $
contains an edge claimed by Toucher, and hence it contains at least
$3$ vertices. Thus both claims follow by induction. 

Let $T'$ be the tree obtained by deleting the vertices $v_{1},\dots,v_{r}$,
and note that $C\subseteq E\left(T'\right)$ as none of the vertices
$v_{1},\dots,v_{r}$ is touched. Hence it follows that $T'$ with
$C'=C$, $D'=\emptyset$ and $X'=\emptyset$ is a reduction of $T$.
Our aim is to construct a suitable sequence of reductions $T_{0},\dots,T_{t}$
for some $t$ with $T_{0}=T'$ and so that $T_{t}$ satisfies $\left|T_{t}\right|-3\left|L_{t}\right|-3\left|C_{t}\right|\geq\left|T\right|-5r-4$.
For each $i$ define $g_{i}=\left|T_{i}\right|-3\left|X_{i}\right|-3\left|C_{i}\right|$. 

The sequence $T_{0},\dots,T_{t}$ is obtained as follows. First of
all, we take $T_{0}=T'$. Given $T_{i-1}$ together with appropriate
sets satisfying $X_{i-1}\subseteq L_{i-1}\subseteq X_{i-1}\cup O_{i-1}$,
we stop the process if $X_{i-1}=L_{i-1}$. Otherwise, there exists
a leaf $v\in T_{i-1}$ with an unique neighbour $w$ satisfying $vw\in C_{i-1}$.
Indeed, this follows by observing that $X_{i-1}\subseteq L_{i-1}$
and $L_{i-1}\setminus X_{i-1}\subseteq O_{i-1}$. 

Given a leaf $v$ with $N\left(v\right)=\left\{ w\right\} $ and $vw\in C_{i-1}$,
we obtain $T_{i}$ as described by one of the cases below, and note
that one of them always occurs given $v$ and $w$ satisfying these
conditions. Note that by our earlier observation such a leaf $v$ certainly
exists in $L_{i-1}\setminus X_{i-1}$, but such a leaf may exist even
in $X_{i-1}$ in which case the reduction can be done as well. Hence
we allow both cases $v\in X_{i-1}$ and $v\not\in X_{i-1}$. We will verify that 
in every case we have $g_{i}\geq g_{i-1}-1$ and
that the property $X_{i}\subseteq L_{i}\subseteq X_{i}\cup O_{i}$
is preserved. \\

\textbf{Case 1. }$w$ satisfies $d_{T_{i-1}}\left(w\right)=2$. \\

Consider $T_{i}$ obtained by deleting the vertex $v$, and setting
$C_{i}=C_{i-1}\setminus\left\{ vw\right\} $, $D_{i}=\emptyset$ and
$X_{i}=\left(X_{i-1}\setminus\left\{ v\right\} \right)\cup\left\{ w\right\} $.
Note that $T_{i}$ is certainly a reduction of $T_{i-1}$ as $w$
is a touched vertex in both $T_{i-1}$ and $T_{i}$. We certainly have
$\left|C_{i}\right|=\left|C_{i-1}\right|-1$ and $\left|T_{i}\right|=\left|T_{i-1}\right|-1$.
Since we might have $v\not\in X_{i-1}$, it follows that $\left|X_{i}\right|\leq\left|X_{i-1}\right|+1$.
Hence we have $g_{i}\geq g_{i-1}-1$, and it is easy to see
that $X_{i}\subseteq L_{i}\subseteq X_{i}\cup O_{i}$. \\

\textbf{Case 2. }$w$ satisfies $d_{T_{i-1}}\left(w\right)\geq3$
and $\left|{\cal E}\left(w\right)\cap C_{i-1}\right|\geq2$. \\

Since $\left|{\cal E}\left(w\right)\cap C_{i-1}\right|\geq2$, it
follows that there exists an edge $uw\in C_{i-1}$ with $u\neq v$.
Consider $T_{i}$ obtained by deleting the vertex $v$, and setting
$C_{i}=C_{i-1}\setminus\left\{ vw\right\} $, $D_{i}=\emptyset$ and
$X_{i}=\left(X_{i-1}\setminus\left\{ v\right\} \right)$. As before,
this is a reduction of $T_{i-1}$ as $w$ is touched in both $T_{i-1}$
and $T_{i}$ since $uw\in C_{i}$. Since $d_{T_{i-1}}\left(w\right)\geq3$,
it follows that $w$ is not a leaf in $T_{i}$, and hence we have
$X_{i}\subseteq L_{i}\subseteq X_{i}\cup O_{i}$. It is easy to check
that $g_{i}\geq g_{i-1}+2$, as required. \\

\textbf{Case 3. }$w$ satisfies $d_{T_{i-1}}\left(w\right)\geq3$
and $\left|{\cal E}\left(w\right)\cap C_{i-1}\right|=1$. \\

Let $N_{T_{i-1}}\left(w\right)\setminus\left\{ v\right\} =\left\{ v_{1},\dots,v_{s}\right\} $.
Since $d_{T_{i-1}}\left(w\right)\geq3$ it follows that $s\geq2$,
and since $\left|{\cal E}\left(w\right)\cap C_{i-1}\right|=1$ it
follows that $wv_{j}\not\in C_{i-1}$ for all $j$. Consider $T_{i}$
obtained by replacing the edge $wv_{1}$ with $vv_{1}$, and by setting
$C_{i}=C_{i-1}$, $D_{i}=\emptyset$ and $X_{i}=X_{i-1}\setminus\left\{ v\right\} $.
Since $vw\in E_{i-1}$ and $T_{i-1}$ is a tree, it follows that $T_{i}$
does not contain a cycle and is connected, and hence $T_{i}$ is also
a tree. Again, $T_{i}$ is a reduction of $T_{i-1}$ by taking $f_{E}\left(wv_{1}\right)=vv_{1}$.
Indeed, this follows from the fact that both $v$ and $w$ are touched
vertices so mapping $wv_{1}$ to $vv_{1}$ satisfies the conditions
of reduction. 

Note that $L_{i}=L_{i-1}\setminus\left\{ v\right\} $, as the only
vertices whose degrees are affected are $v$ and $w$, and since $s\geq2$
it follows that $w$ is not a leaf in $T_{i}$. It is easy to see
that $\left|T_{i}\right|=\left|T_{i-1}\right|$, $\left|X_{i}\right|\leq\left|X_{i-1}\right|$
and $\left|C_{i}\right|=\left|C_{i-1}\right|$. Hence it follows that
$g_{i}\geq g_{i-1}$, and it is also easy to see that $X_{i}\subseteq L_{i}\subseteq X_{i}\cup O_{i}$.
\\

Note that we still need to verify that any sequence of such operations
will terminate in a finite time. In every application of Case 1 the
number of vertices in $T_{i}$ decreases by $1$, yet the size of
$T_{i}$ remains unaffected in Cases 2 and 3. Thus Case 1 can be applied
at most $\left|T_{0}\right|$ times. On the other hand, the number
of leaves in $T_{i}$ decreases by 1 in every application of Cases
2 or 3, so the number of times Cases 2 or 3 can be applied consecutively
without applying Case 1 is at most the number of vertices at that
particular stage. Hence the total number of applications is at most
$\sum_{i=1}^{\left|T_{0}\right|}i$, which proves that the process
must terminate in a finite time. 

Let $T_{0},\dots,T_{t}$ be the sequence of reductions obtained during 
the process, and let $a$ be the number of times Case 1 is applied.
Since $g_{i}\geq g_{i-1}-1$ whenever Case 1 is applied, $g_{i}\geq g_{i-1}+2$
whenever Case 2 is applied and $g_{i}\geq g_{i-1}$ whenever Case
3 is applied, it follows that $g_{t}\geq g_{0}-a$. 

On the other hand, note that $\left|C_{i}\right|=\left|C_{i-1}\right|$
whenever Case 3 is applied, yet $\left|C_{i}\right|=\left|C_{i-1}\right|-1$
whenever Case 1 or 2 is applied. Since $\left|C_{t}\right|\geq0$
and $\left|C_{0}\right|=r+1$, it follows that $a\leq r+1$. Thus
we must have $g_{t}\geq g_{0}-\left(r+1\right)$. Since $X_{t}=L_{t}$
and $X_{0}=\emptyset$, this can be rewritten as 
\[
\left|T_{t}\right|-3\left|L_{t}\right|-3\left|C_{t}\right|\geq\left|T_{0}\right|-3\left|C_{0}\right|-\left(r+1\right).
\]
Since $\left|T_{0}\right|=\left|T\right|-r$ and $\left|C_{0}\right|=r+1$,
it follows that 
\[
\left|T_{t}\right|-3\left|L_{t}\right|-3\left|C_{t}\right|\geq\left|T\right|-5r-4.
\]
Since $T_{t}$ is a reduction of $T$, this completes the proof. 
\end{proof}

\section{Delayed version of the game }

Let $T_{t}$, $C_{t}$ and $L_{t}$ be given by Lemma \ref{lem:Lemma 3}
and let $r$ be the number of edges claimed by Isolator during the
first phase of the game. Then Lemma \ref{lem:Lemma 2} implies that
$\alpha\left(T\right)\geq r+\alpha\left(T_{t},C_{t},L_{t}\right)$.
Since $\left|T_{t}\right|-3\left|L_{t}\right|-3\left|C_{t}\right|\geq\left|T\right|-5r-4$,
it suffices to prove that for all trees $T$ with $n$ vertices, $l$
leaves and for any set of edges $C\subseteq E$ we have $\alpha\left(T,C,L\right)\geq\left\lfloor \frac{n-3l-3\left|C\right|+7}{5}\right\rfloor $. 

In order to make our inductive proof work, we need to prove a slightly
stronger statement. Recall that $O$ is the set of occupied vertices,
and since $X=L$ it follows that $O$ is the set of those vertices
of degree at least $2$ which are endpoint of an edge in $C$. Our
aim is to prove the following result. 
\begin{lem}
\label{lem:Lemma 4}Let $T$ be a tree with $n$ vertices and $l$
leaves. Let $O$ be the set of occupied vertices of $T$ and let $C\subseteq E$.
Then 
\begin{equation}
\alpha\left(T,C,L\right)\geq\left\lfloor \frac{n-3l-3\left|C\right|+7+\sum_{v\in O}\left(d\left(v\right)-2\right)}{5}\right\rfloor .\label{eq:2-1}
\end{equation}
In particular, it follows that $\alpha\left(T,C,L\right)\geq\left\lfloor \frac{n-3l-3\left|C\right|+7}{5}\right\rfloor .$
\end{lem}

The proof is an inductive proof, first on the size on $T$ and then on
the number of leaves in $T$. However, for simplicity one could view
it just as an inductive proof on the size of $T$. Given a tree $T$
whose leaves are touched and with some edges claimed by Toucher at
the start, our aim is to either find suitable edges for Isolator that
can help to isolate some vertices, or find suitable substructures
of the tree with many touched edges that could be removed without
deleting too many vertices that could possibly be isolated. In either
case the aim is to reduce the size of $T$ without reducing the lower
bound for the score. 

The structures we are in general looking for are neighbouring vertices
of degree $1$ or $2$, as near such vertices $T$ behaves similarly
to path. If no such substructures exists, it follows that vertices
of degree $1$ or $2$ must be spread out. In particular, it follows
that there must be vertices of higher degree, which in turn implies
that $T$ has plenty of leaves. The aim of the next Lemma is to make
this argument precise. As a consequence, it turns out that if no suitable
substructure of $T$ exist, then the expression $\left\lfloor \frac{n-3l-3\left|C\right|+7+\sum_{v\in O}\left(d\left(v\right)-2\right)}{5}\right\rfloor $
in (\ref{eq:2-1}) turns out to be at most $0$, and hence the claim
is certainly true.

Since there are several substructures we are considering in $T$,
the proof splits into many cases and the proofs of some cases are
rather long. This is due to the fact that for each substructure, the
proof often splits into multiple sub-cases based on the structure
of $T$ on vertices near the substructure. In general, the proofs
are fairly easy within each case, and the same ideas are repeatedly in
different cases. In a sense, the hardest idea is to come up with a
suitable lower bound in (\ref{eq:2-1}) that is strong enough for
an inductive argument. 
\begin{lem}
\label{lem:Lemma5}Let $T$ be a tree with $n\geq3$ vertices in which
there are no two adjacent vertices of degree $2$ and no leaf adjacent
to a vertex of degree $2$. Then $T$ contains at least $\frac{n+5}{3}$
leaves. 
\end{lem}

\begin{proof}
Let $d_{i}$ be the number of vertices in $T$ of degree $i$. It
is well-known that we have 
\begin{equation}
\sum_{i=1}^{n}id_{i}=2\left(n-1\right)\label{eq:L1.1}
\end{equation}
and 
\begin{equation}
d_{1}=2+\sum_{i=3}^{n}\left(i-2\right)d_{i}.\label{eq:L1.2}
\end{equation}

Let $X$ be the set of vertices in $T$ that have degree $1$ or $2$,
and let $Y$ be the set of vertices in $T$ that have degree at least
$3$. Note that there are no edges inside $X$. Indeed, trivially
no two leaves can be adjacent in any tree with at least $3$ vertices,
and by assumption no leaf is adjacent to a vertex of degree $2$ and
no two vertices of degree $2$ are adjacent. Hence it follows that
\begin{equation}
d_{1}+2d_{2}=e\left(X,Y\right)\leq\sum_{y\in Y}d\left(y\right)=\sum_{i=3}^{n}id_{i}.\label{eq:L1.3}
\end{equation}
Combining (\ref{eq:L1.1}) with (\ref{eq:L1.3}), it follows that
\[
\sum_{i=3}^{n}id_{i}\geq n-1.
\]
Since $3\left(i-2\right)\geq i$ holds for all $i\geq3$, it follows
that 
\[
\sum_{i=3}^{n}\left(i-2\right)d_{i}\geq\frac{1}{3}\sum_{i=3}^{n}id_{i}\geq\frac{n-1}{3}.
\]
Thus from (\ref{eq:L1.2}) it follows that $d_{1}\geq2+\frac{n-1}{3}=\frac{n+5}{3}$,
which completes the proof. 
\end{proof}
We are now ready to prove Lemma \ref{lem:Lemma 4}. 

\begin{proof}[Proof of Lemma \ref{lem:Lemma 4}] The proof is by induction
on $n$, and for fixed $n$ we also induct on the number of leaves.
Let $C=\left\{ e_{1},\dots,e_{k}\right\} $ be the set of edges claimed
by Toucher at the start of the game, and for convenience let $l=\left|L\right|$
and $k=\left|C\right|$ throughout the proof. We start by checking
the base cases. For fixed $n$, note that the claim follows if $T$
is path by \cite{key-9}. Hence for given $n$, the base case for
the induction on the number of leaves hold. Thus we may always assume
that $T$ is not a path. 

Next we prove that the claim holds whenever $n\leq5$. Since $T$
is not a path, we must have $l\geq3$. Note that we always have $\sum_{v\in O}\left(d\left(v\right)-2\right)\leq\sum_{v\not\in L}\left(d\left(v\right)-2\right)=l-2$.
Hence it follows that for $n\leq5$ and $l\geq3$ we have 
\[
n+7-3l-3k+\sum_{v\in O}\left(d\left(v\right)-2\right)\leq n+5-2l\leq4,
\]
which completes the proof as we always have $\alpha\left(T,C,L\right)\geq0$.
Thus from now on we may assume that $T$ has at least $6$ vertices
and that $T$ is not a path.

We split the proof into cases based on whether $T$ contains suitable
substructures. At the end we prove that if $T$ contains none of these
substructures, then we must have $n-3l-3k+7+\sum_{v\in O}\left(d\left(v\right)-2\right)<5$,
in which case the claim follows as well. 

In the first four cases we consider those situations in which $C$
contains two edges that are 'close' to each others or an edge close
to a leaf. In those cases we prove that a suitable part of the tree
can be removed already before the start of the game in a way that
the resulting tree is a reduction of the original tree and so that
this reduction does not decrease the score. In particular, note that
we have $D=I=\emptyset$ in those cases. In the remaining two cases
we consider situations when $T$ contains sufficiently many neighbouring
vertices of degree $2$ which can be claimed by Isolator in order
to increase the score. 

Let us first focus on those cases in which we can simply reduce $T$
before the game starts. Given a reduction $T_{1}$ of $T$ with appropriate
sets $C_{1}$, $L_{1}$ and $D_{1}$ satisfying $D_{1}=\emptyset$,
for convenience we define $d\left(k\right)=\left|C\right|-\left|C_{1}\right|$,
$d\left(l\right)=\left|L\right|-\left|L_{1}\right|$, $d\left(n\right)=\left|T\right|-\left|T_{1}\right|$
and 
\[
d\left(s\right)=\sum_{v\in O}\left(d_{T}\left(v\right)-2\right)-\sum_{v\in O_{1}}\left(d_{T_{1}}\left(v\right)-2\right).
\]
For $v\in V\left(T\right)$, define 
\[
d_{s}\left(v\right)=\left(d_{T}\left(v\right)-2\right)\mathbb{I}\left\{ v\in O\right\} -\left(d_{T_{1}}\left(v\right)-2\right)\mathbb{I}\left\{ v\in O_{1}\right\} ,
\]
where $\mathbb{I}$ denotes the indicator function of an event, and
note that $d\left(s\right)=\sum_{v\in V\cup V_{1}}d_{s}\left(v\right)$.
Finally define $D\left(T,T_{1}\right)=d\left(n\right)-3d\left(l\right)-3d\left(k\right)+d\left(s\right)$.
Note that this also depends on the sets $C$ and $C_{1}$, but the
dependence will not be highlighted in the notation as these sets are
clear from the context. 

Let 
\[
S\left(T\right)=\left\lfloor \frac{n-3l-3k+7+\sum_{v\in O}\left(d\left(v\right)-2\right)}{5}\right\rfloor ,
\]
and again note that $S$ also depends on $C$. If $d\left(n\right)>0$,
then the inductive hypothesis implies that $\alpha\left(T_{1},C_{1},L_{1}\right)\geq S\left(T_{1}\right)$.
If we also had $D\left(T,T_{1}\right)\leq0$, it would certainly follow
that $S\left(T_{1}\right)\geq S\left(T\right)$. Since $T_{1}$ is
a reduction of $T$, Lemma \ref{lem:Lemma 2} implies that $\alpha\left(T,C,L\right)\geq\alpha\left(T_{1},C_{1},L_{1}\right)$, 
as $D=\emptyset$ implies that $I=\emptyset$. Combining all of
these together implies that $\alpha\left(T,C,L\right)\geq S\left(T\right)$,
as required. Hence if $\left|T\right|>\left|T_{1}\right|$ it suffices
to prove that $D\left(T,T_{1}\right)\leq0$. We now move on to considering
various substructures of $T$. \\

\textbf{Case 1. }$T$ contains an unoccupied vertex of degree $2$
whose both neighbours are touched.\\

Since $n\geq3$, it follows that either both of the vertices are occupied,
or one of them is occupied and the other is a leaf. We start by reducing
the first case to the second case. 

Let $v$ be the unoccupied vertex of degree $2$, and let $u$ and
$w$ be the neighbours of $v$, and set $N\left(u\right)\setminus\left\{v\right\} =\left\{ u_{1},\dots,u_{r}\right\} $. 
Since $u$ is occupied, it follows that there exists $j$ for which $uu_{j}\in C$. Let $a$ be a leaf in $T$ so that every
path from $a$ to $u$ must go through $w$. Consider $T_{1}$ obtained
by deleting all the edges with $u$ as an endpoint apart from the edge
$uv$ and adding the edges $au_{1},\dots,au_{r}$, as illustrated
in Figure 1. We also take $C_{1}$ to be the set containing all the edges in $C$ that do not have $u$ as an endpoint,
and all the edges of the form $au_{i}$ for those $i$ with $uu_{i}\in C$.

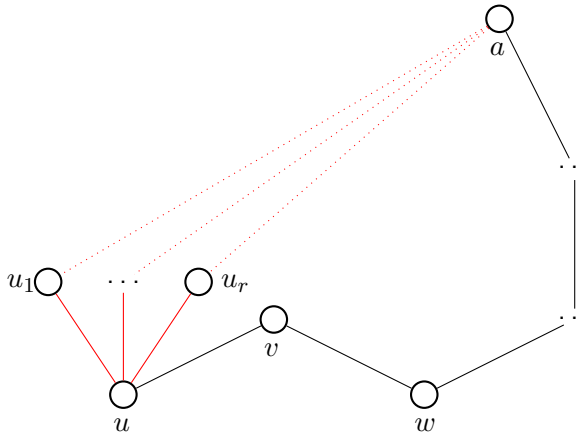
\begin{figure}
\centering
\caption{Construction of $T_{1}$. Red rigid edges are edges in $T$ which
are replaced with red dotted edges. }

\begin{tikzpicture}

\node[main node] (n1) at (1,8){};
\node[text width=0.5cm] at (1.12,7.6){$u$};
\node[main node] (n2) at (3,9){};
\node[text width=0.5cm] at (3.12,8.6){$v$};
\node[main node] (n3) at (5,8){};
\node[text width=0.5cm] at (5.12,7.6){$w$};
\node[every node](n4) at (7,9){};
\node[every node](n9) at (7,11){};
\node[text width=0.5cm] at (7.025,9.025){\ldots}; 
\node[text width=0.5cm] at (7.025,11.025){\ldots};
\node[text width=0.5cm] at (1.025,9.505){\ldots};
\node[main node] (n5) at (6,13){};
\node[text width=0.5cm] at (6.12,12.6){$a$};

\node[main node] (n6) at (0,9.5){};
\node[text width=0.5cm] at (-0.3,9.48){$u_1$};
\node[every node ] (n7) at (1,9.5){};
\node[main node] (n8) at (2,9.5){};
\node[text width=0.5cm] at (2.55,9.48){$u_r$};

\draw (n1) -- (n2);
\draw (n2) -- (n3);
\draw (n3) -- (n4);
\draw (n4) -- (n9);
\draw (n9) -- (n5);
\draw[style={draw=red}] (n1) -- (n6);
\draw[style={draw=red}] (n1) -- (n7);
\draw[style={draw=red}] (n1) -- (n8);

\draw[style={draw=red, dotted}]  (n5)--(n6);
\draw[style={draw=red, dotted}]  (n5)--(n7);
\draw[style={draw=red, dotted}]  (n5)--(n8);

\end{tikzpicture}

\end{figure}

It is easy to see that $T_{1}$ is a reduction of $T$ by taking  $f_{E}\left(uu_{i}\right)=au_{i}$. 
Indeed, this follows from the fact that $a$ is touched in $T_{1}$ as $au_{j}\in C_{1}$.
It is also easy to check that $d\left(k\right)=0$ and $d\left(n\right)=0$.
We also have $L_{1}=\left(L\setminus\left\{ a\right\} \right)\cup\left\{ u\right\} $,
which implies that $d\left(l\right)=0$. Note that $a$ and $u$ 
are the only vertices whose degrees are affected during the process.
Since we have $d_{s}\left(a\right)=1-r$ and $d_{s}\left(u\right)=r-1$, it follows that $d\left(s\right)=0$. In
particular, we have $S\left(T\right)=S\left(T_{1}\right)$, and thus
by Lemma \ref{lem:Lemma 2} it suffices to prove the claim for $T_{1}$. 

Hence we may assume that $T$ contains an unoccupied vertex $v$ of
degree $2$ with neighbours $u$ and $w$ so that $u$ is a leaf and
$w$ is occupied. Let $w_{1}$ be chosen such that $ww_{1}\in C$.
Since $T$ has at least $6$ vertices, we must have $d\left(w\right)+d\left(w_{1}\right)\geq4$.
We start by considering the cases corresponding to $d\left(w\right)+d\left(w_{1}\right)=4$,
and it is easy to check that this occurs exactly when $\left(d\left(w\right),d\left(w_{1}\right)\right)\in\left\{ \left(2,2\right),\left(3,1\right)\right\} $. 

If $d\left(w\right)=d\left(w_{1}\right)=2$, consider $T_{1}$ obtained
by deleting the vertices $u$, $v$ and $w$, and take $C_{1}=C\setminus\left\{ ww_{1}\right\} $.
Since $w_{1}$ is a leaf in $T_{1}$, it follows that $T_{1}$ is
a reduction of $T$. Since $L_{1}=\left(L\setminus\left\{ v\right\} \right)\cup\left\{ w_{1}\right\} $,
it follows that $d\left(l\right)=0$. It is also easy to check that
we have $d\left(n\right)=3$, $d\left(k\right)=1$ and $d\left(s\right)=0$.
Thus it follows that $D\left(T,T_{1}\right)=0$, and since $\left|T\right|>\left|T_{1}\right|$
the claim follows by induction. 

If $d\left(w\right)=3$ and $d\left(w_{1}\right)=1$, let $x$ be
chosen such that $N\left(w\right)=\left\{ x,v,w_{1}\right\} $. Consider
$T_{1}$ obtained by deleting the vertices $u$, $v$ and $w_{1}$,
and take $C_{1}=C\setminus\left\{ ww_{1}\right\} $. Since $N_{T}\left(w\right)=\left\{ x,v,w_{1}\right\} $,
it follows that $w$ is a leaf in $T_{1}$, and hence $T_{1}$ is
a reduction of $T$. Note that we have $L_{1}=\left(L\setminus\left\{ u,w_{1}\right\} \right)\cup\left\{ w\right\} $
and that $w$ is the only vertex whose degree is affected during the 
process. Thus it follows that $d\left(n\right)=3$, $d\left(l\right)=1$
and $d\left(k\right)=1$, and since $d_{s}\left(w\right)=1$ we also
have $d\left(s\right)=1$. Hence we have $D\left(T,T_{1}\right)=-2$,
and since $\left|T\right|>\left|T_{1}\right|$ the claim follows by
induction. 

Now suppose that $d\left(w\right)+d\left(w_{1}\right)\geq5$. Let
$N\left(w_{1}\right)\setminus\left\{ w\right\} =\left\{ a_{1},\dots,a_{c}\right\} $
and $N\left(w\right)\setminus\left\{ v,w_{1}\right\} =\left\{ a_{c+1},\dots,a_{d}\right\} $
where one of these sets might be empty. Note that $d\left(w\right)+d\left(w_{1}\right)\geq5$
implies that $d\geq2$. Consider $T_{1}$ obtained by deleting the
vertices $u$ and $v$ and by taking $E\left(T_{1}\right)$ to be
the set of those edges in $T$ that do not have $w$ or $w_{1}$ as
their endpoint together with the edges $ww_{1}$, $w_{1}a_{1}$ and
$wa_{i}$ for $2\leq i\leq d$. See Figure 2 for illustration when
$c=0$. Finally we take $C_{1}$ to be the set containing the edge
$ww_{1}$, all the edges in $C$ that do not have $w$ or $w_{1}$
as an endpoint, and the unique edge in $\left\{ wa_{i},w_{1}a_{i}\right\} \cap E\left(T_{1}\right)$
for those $i$ for which the one of $wa_{i}$ or $w_{1}a_{i}$ that is
an edge in $T$ is also in $C$. In particular, it follows that $\left|C_{1}\right|=\left|C\right|$.

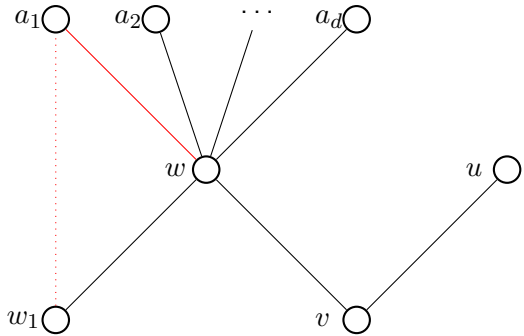
\begin{figure}
\centering
\caption{Construction of $T_{1}$ when $d\left(w\right) + d\left(w_{1}\right)\geq 5$ and $c=0$. Red rigid edge is an edge in
$T$ that is replaced with the red dotted edge. }
\begin{tikzpicture}
\node[main node] (n1) at (1,10){};
\node[text width=0.5cm] at (0.7,10){$w$};
\node[main node] (n2) at (3,8){};
\node[text width=0.5cm] at (2.7,8){$v$};
\node[main node] (n3) at (5,10){};
\node[text width=0.5cm] at (4.7,10){$u$};
\node[main node] (n5) at (-1,8){};
\node[text width=0.5cm] at (-1.4,8){$w_1$};
\node[main node](n6) at (-1,12){};
\node[text width=0.5cm] at (-1.3,12){$a_1$};
\node[main node](n10) at (0.33,12){};
\node[text width=0.5cm] at (0.03,12){$a_2$};
\node[every node](n7) at (1.66,12){};
\node[main node](n8) at (3,12){};
\node[text width=0.5cm] at (2.7,12){$a_d$};
\node[text width=0.5cm] at (1.7,12.1){\ldots}; 

\draw (n1) -- (n2);
\draw (n2) -- (n3);
\draw (n1) -- (n5);
\draw (n1) -- (n7);
\draw (n1) -- (n8);
\draw (n1) -- (n10);
\draw[style={draw=red}] (n1) -- (n6);
\draw[style={draw=red, dotted}]  (n5)--(n6);

\end{tikzpicture}

\end{figure}

Note that $T_{1}$ is a reduction of $T$ by taking $f\left(wa_{i}\right)=a_{1}a_{i}$
for all $2\leq i\leq c$ since both $w$ and $w_{1}$ are touched
in $T$ and $T_{1}$. It is easy to check that we have $d\left(n\right)=2$,
$d\left(l\right)\geq1$ and $d\left(k\right)=0$. It is easy to see
that the only vertices whose degrees are affected are $w$ and
$w_{1}$. We have $d_{s}\left(w\right)=\left(d-c\right)-\left(d-1\right)=1-c$
and $d_{s}\left(w\right)=\max\left(c-1,0\right)-0\leq c$ as $c\geq0$. Also note that $u$ and $v$ are the only 
deleted vertices, and we have $d_{s}\left(v\right)=0$ and $d_{s}\left(u\right)=0$. 
Hence it follows that $d\left(s\right)\leq1$ and hence it
follows that $D\left(T,T_{1}\right)\leq0$. Since $\left|T\right|>\left|T_{1}\right|$,
the claim follows by induction. This completes the proof of Case 1.
$\hfill\square$\\

\textbf{Case 2.} $T$ contains an edge $e\not\in C$ whose both endpoints
are touched. \\

There are again two possibilities: either both endpoints of $e$ are
occupied or one of them is occupied and the other is a leaf. By using
same argument as in Case 1 we may assume that one of the endpoints
is occupied and the other is a leaf. Let $u$ and $v$ be the endpoints
of $e$ such that $u$ is a leaf and $v$ is occupied, and let $w$
be chosen such that $vw\in C.$ Again we split into cases based on
the size of $d\left(v\right)+d\left(w\right)$, and since $T$ has
at least $6$ vertices it follows that $d\left(v\right)+d\left(w\right)\geq4$,
and again $d\left(v\right)+d\left(w\right)=4$ implies that $\left(d\left(v\right),d\left(w\right)\right)\in\left\{ \left(2,2\right),\left(3,1\right)\right\} $. 

If $d\left(v\right)=d\left(w\right)=2$, consider $T_{1}$ obtained
by deleting the vertices $u$ and $v$, and take $C_{1}=C\setminus\left\{ vw\right\} $.
Since $d_{T}\left(w\right)=2$, it follows that $w$ is a leaf in
$T_{1}$, and hence $T_{1}$ is a reduction of $T$. It is easy to
check that we have $d\left(n\right)=2$, $d\left(l\right)=0$ and
$d\left(k\right)=$1. Since $w$ is the only vertex whose degree is
affected during the process and $d_{s}\left(w\right)=0$, it follows
that $d\left(s\right)=0$. Hence we have $D\left(T,T_{1}\right)=-1$,
and since $\left|T\right|>\left|T_{1}\right|$ the claim follows by
induction. 

If $d\left(v\right)=3$ and $d\left(w\right)=1$, let $T_{1}$ be
the tree obtained by deleting the vertices $u$ and $w$, and set 
$C_{1}=C\setminus\left\{ vw\right\} $. Since $v$ is a leaf in $T_{1}$
it follows that $T_{1}$ is a reduction of $T$. It is easy to check
that $d\left(n\right)=2$, $d\left(l\right)=1$ and $d\left(k\right)=1$.
Note that $v$ is the only vertex whose degree is affected during
the process, and since $d_{s}\left(v\right)=1$ it follows that $d\left(s\right)=1$.
Hence we have $D\left(T,T_{1}\right)=-3$, and since $\left|T\right|>\left|T_{1}\right|$
the claim follows by induction. 

Finally suppose that $d\left(v\right)+d\left(w\right)\geq5$. Let
$N\left(v\right)\setminus\left\{ u,w\right\} =\left\{ a_{1},\dots,a_{c}\right\} $
and $N\left(w\right)\setminus\left\{ v\right\} =\left\{ a_{c+1},\dots,a_{d}\right\} $
where one of these sets might be empty. Note that $d\left(v\right)+d\left(w\right)\geq5$
implies that $d\geq2$. Consider $T_{1}$ obtained by deleting the
vertex $u$, and we take $E\left(T_{1}\right)$
to be the set of those edges in $T$ that do not have $v$ or $w$
as their endpoint together with $vw$, $va_{1}$ and $wa_{i}$ for
$2\leq i\leq d$. Let $C_{1}$ to be the set containing the
edge $vw$, all the edges in $C$ that do not have $v$ or $w$ as
an endpoint, and the unique edge in $\left\{ va_{i},wa_{i}\right\} \cap E\left(T_{1}\right)$
for those $i$ for which the one of $va_{i}$ or $wa_{i}$ that is
an edge in $T$ is also in $C$, and note that we have $\left|C_{1}\right|=\left|C\right|$.
Since $v$ and $w$ are occupied in both $T$ and $T_{1}$,
it follows that $T_{1}$ is a reduction of $T$ by taking $f_{E}\left(va_{i}\right)=wa_{i}$
for all $2\leq i\leq c$ if $c\geq1$, or by taking $f_{E}\left(wa_{1}\right)=va_{1}$ if $c=0$. 

Note that we have $d\left(n\right)=1$, $d\left(l\right)\geq1$ and
$d\left(k\right)=0$. Since $v$ and $w$ are the only vertices whose
degrees are affected during the process, and since we have $d_{s}\left(v\right)=c$
and $d_{s}\left(w\right)=\max\left(d-c-1,0\right)-\left(d-2\right)\leq2-c$,
it follows that $d\left(s\right)\leq2$. Hence we have $D\left(T,T_{1}\right)\leq0$,
and since $\left|T\right|>\left|T_{1}\right|$ the claim follows by
induction. This completes the proof of Case 2. $\hfill\square$\\

\textbf{Case 3}. There exist an edge $e\in C$ whose endpoint is a
leaf. \\

Let $u$ and $v$ be the endpoints of $e$ with $u$ being the leaf.
First suppose that $d\left(v\right)=2$, and let $w$ be the other
neighbour of $v$. Let $T_{1}$ be obtained by deleting the vertex
$u$, and take $C_{1}=C\setminus\left\{ uv\right\} $. Since $v$
is touched in both $T_{1}$ and $T$, it follows that $T_{1}$ is a
reduction of $T$. It is easy to check that we have $d\left(n\right)=1$,
$d\left(l\right)=0$, $d\left(k\right)=1$ and $d\left(s\right)=0$.
Thus it follows that $D\left(T,T_{1}\right)=-2$, and since $\left|T\right|>\left|T_{1}\right|$
the claim follows by induction. 

Now suppose that $d\left(v\right)\geq3$, and let $N\left(v\right)\setminus\left\{ u\right\} =\left\{ v_{1},\dots,v_{c}\right\} $
where $c\geq2$. Consider $T_{1}$ obtained by replacing the edge
$vv_{1}$ with $uv_{1}$. It is easy to see that $T_{1}$ is a reduction
of $T$ by taking $f_{E}\left(vv_{1}\right)=uv_{1}$, as $u$ and
$v$ are touched vertices in $T_{1}$ and $T$. It is easy to check
that $d\left(n\right)=0$, $d\left(l\right)=1$ and $d\left(k\right)=0$.
Note that $u$ and $v$ are the only vertices whose degrees are affected
during the process, and we clearly have $d_{s}\left(v\right)=1$
and $d_{s}\left(u\right)=0$. Hence it follows that $d\left(s\right)=1$,
and thus $D\left(T,T_{1}\right)=-2\leq0$. Since the number of vertices
remains the same and the number of leaves decreases by one, the claim
follows by induction. This completes the proof of Case 3. $\hfill\square$
\\

\textbf{Case 4. }There exist distinct edges $e_{i},\,e_{j}\in C$
that have a common endpoint. \\

Let $u$ be the common endpoint of $e_{i}$ and $e_{j}$, and let
$v$ and $w$ be the other endpoints respectively. By Case 3 we may
assume that neither of $v$ nor $w$ is a leaf. Let $N\left(u\right)\setminus\left\{ v,w\right\} =\left\{ u_{1},\dots,u_{r}\right\} $
with possibly $r=0$. Consider $T_{1}$ obtained by removing the vertex
$u$ together with all the edges that have $u$ as an endpoint,
and by adding the edges $vw$ and $vu_{i}$ for all $1\leq i\leq r$ as in Figure 3, 
and note that $T_{1}$ is certainly a tree. Let $C_{1}$ be the
set containing the edge $vw$, all the edges in $C$ that do not have
$u$ as an endpoint and all the edges $vu_{i}$ for those $i$ for
which $uu_{i}\in C$. Since both $u$ and $v$ are touched, it follows
that $T_{1}$ is a reduction of $T$ by taking $f_{E}\left(uu_{i}\right)=vu_{i}$
for all $i$. 

\begin{figure}
\centering
\caption{Construction of $T_{1}$. Red rigid edges are edges in $T$ which
are replaced with red dotted edges.}
\begin{tikzpicture}

\node[main node] (n1) at (1,6){};
\node[main node] (n2) at (2,5){};
\node[main node] (n3) at (3,6){};
\node[every node] (n4) at (-0.4,6){};
\node[every node] (n5) at (4.4,6){};
\node[text width=0.5cm] at (-0.5,6){\ldots};
\node[text width=0.5cm] at (4.5,6){\ldots};
\node[main node] (n6) at (4,4){};
\node[main node] (n7) at (2,4){};
\node[every node] (n8) at (3,4){};
\node[text width=0.5cm] at (2.95,3.95){\ldots};
\node[text width=0.5cm] at (1.15, 6.3){$w$};
\node[text width=0.5cm] at (3.15, 6.3){$v$};
\node[text width=0.5cm] at (2.15, 5.3){$u$};
\node[text width=0.5cm] at (2.15, 3.65){$u_1$};
\node[text width=0.5cm] at (4.15, 3.65){$u_r$};

\draw[style={draw=red}] (n1) -- (n2);
\draw[style={draw=red}] (n2) -- (n3);
\draw (n1) -- (n4);
\draw (n3) -- (n5);
\draw[style={draw=red}] (n2) -- (n6);
\draw[style={draw=red}] (n2) -- (n7);
\draw[style={draw=red}] (n2) -- (n8);
\draw[style={draw=red}, dotted] (n1) -- (n3);
\draw[style={draw=red}, dotted] (n3) -- (n6);
\draw[style={draw=red}, dotted] (n3) -- (n7);
\draw[style={draw=red}, dotted] (n3) -- (n8);
\end{tikzpicture}
\end{figure}
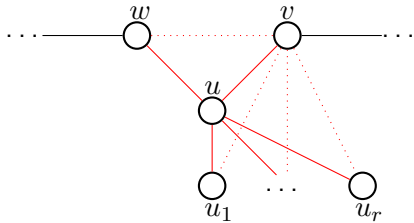

It is easy to check that $d\left(n\right)=1$, $d\left(l\right)=0$
and $d\left(k\right)=1$. Note that the only vertices whose degrees
are affected during the process are $v$ and $w$, and the vertex $u$ is also deleted. Since neither
of $v$ nor $w$ is a leaf, it is easy to check that $d_{s}\left(u\right)=r$,
$d_{s}\left(v\right)=\left(d_{T}\left(v\right)-2\right)-\left(d_{T}\left(v\right)+r-1-2\right)=1-r$
and $d_{s}\left(w\right)=1$. In particular, it follows that $d\left(s\right)=2$
and hence we have $D\left(T,T_{1}\right)=0$. Since $\left|T\right|>\left|T_{1}\right|$,
the claim follows by induction, and this completes the proof of Case
4. $\hfill\square$\\

From now on we suppose that $T$ contains no configuration described
in Cases 1-4, and hence the edges in $C$ are suitably 'isolated'.
We now consider those cases when $T$ contain two adjacent vertices
of degree $2$ that are both unoccupied. In such a case the edges
incident to the vertex of degree $2$ could be suitable moves for Isolator.
Our aim is to describe a sequence of moves for Isolator that allow
her to increase the score in a way that the resulting tree (together
with the new moves) has a reduction with sufficiently large score. 

From now on we change our notation slightly: let $C'$ be the set
of those edges claimed by Toucher at the start of the delayed game
and let $O'$ be the set of occupied vertices at the start of the
delayed game. Let $D$ be the set of those edges claimed by Isolator
during the new moves, and let $\hat{C}$ be the set of edges claimed
by Toucher during the new moves, and for convenience we write $\hat{C}=\left\{ f_{1},\dots,f_{\left|D\right|}\right\} $.
Finally we set $C=C'\cup\hat{C}$, and hence $C$ is the set of all
edges claimed by Toucher at the end of the process, i.e.\ once the
new moves have been played. 

In all the cases we are about to consider, the edges in $D$ form
a path in $T$ so that all the vertices on this path except the endpoints
have degree $2$ in $T$, and the endpoints have degree at least $3$
or are touched at the end of the process. In particular, it follows
that the number of vertices isolated during the process is exactly $\left|D\right|-1$. 

Again our aim is to seek for a suitable reduction $T_{1}$ of $T$
for these choices of $C$, $D$ and $X=L$, and as usual we require
that $D_{1}=\emptyset$ and $X_{1}=L_{1}$. Again, since there are
exactly $\left|D\right|-1$ isolated vertices, Lemma \ref{lem:Lemma 2}
implies that $\alpha\left(T,C,D,L\right)\geq\left(\left|D\right|-1\right)+\alpha\left(T_{1},C_{1},D_{1},L_{1}\right)$.
If $\left|T\right|>\left|T_{1}\right|$, we know that $\alpha\left(T_{1},C_{1},D_{1},L_{1}\right)\geq S\left(T_{1}\right)$
by induction. Note that during the process of claiming new edges we
fix a suitable strategy for Isolator, but we allow Toucher to play
arbitrary edges on her moves. Hence it follows that $\alpha\left(T,C',L\right)\geq\alpha\left(T,C,D,L\right)$,
as playing the edges in $D$ corresponds to a certain choice of strategy,
which may or may not be optimal. 

Define $d\left(n\right)=\left|T\right|-\left|T_{1}\right|$, $d\left(l\right)=\left|L\right|-\left|L_{1}\right|$,
$d\left(k\right)=\left|C\right|-\left|C'\right|$ and 
\[
d\left(s\right)=\sum_{v\in O'}\left(d\left(v\right)-2\right)-\sum_{v\in O}\left(d\left(v\right)-2\right).
\]
In particular, note that $d\left(k\right)$ and $d\left(s\right)$
are defined for the initial set-up of the delayed game, and not for
the set-up containing the new edges that are played. As before, we
define $D\left(T,T_{1}\right)=d\left(n\right)-3d\left(l\right)-3d\left(k\right)+d\left(s\right)$.
Again, if we had $D\left(T,T_{1}\right)\leq5\left(\left|D\right|-1\right)$,
it would follow that $S\left(T_{1},C_{1}\right)+\left(\left|D\right|-1\right)\geq S\left(T,C'\right)$
and hence it would follow that $\alpha\left(T,C',L\right)\geq S\left(T,C'\right)$,
where the dependence on the set of claimed edges $C'$ is highlighted
in the notation for clarity. Hence if $\left|T\right|>\left|T_{1}\right|$
it suffices to prove that $D\left(T,T_{1}\right)\leq5\left(\left|D\right|-1\right)$. 

We now focus on the unoccupied vertices of degree $2$. First we consider
a case when there exists such an unoccupied vertex whose neighbour
is a touched vertex, although this case splits into a number of sub-cases. \\

\textbf{Case 5.} There exists an unoccupied vertex of degree $2$
whose neighbour is touched. \\

Let $v_{1}$ be the unoccupied vertex of degree $2$ and let $v_{0}$
be the neighbour of $v_{1}$ that is touched. Note that by using the
same argument as in the Case 1 we may assume that $v_{0}$ is a leaf.
We start by constructing a sequence of vertices $v_{0},v_{1},\dots,v_{m}$
as follows: given an unoccupied vertex $v_{i}$ of degree $2$ with
$v_{i-1}v_{i}\in E$, let $v_{i+1}$ be chosen such that $N\left(v_{i}\right)=\left\{ v_{i-1},v_{i+1}\right\} $.
Let $m$ denote the index for which the process stops, i.e.\ $m$
is the least positive integer for which $v_{m}$ is touched or $d\left(v_{m}\right)\geq3$.
Since $v_{1}$ is an unoccupied vertex of degree $2$, it follows
that $m\geq2$. 

We now split into our main cases which mostly depend on the value
of $m$. For convenience, we say that a vertex $v$ is \textit{initially
touched} if $v$ is touched in the initial set-up of the board. That
is, if $v$ is unoccupied on the initial board but $v$ is an endpoint
of one of the move played by Toucher once the game has started, we
do not consider $v$ as an initially touched vertex and we say that
$v$ is \textit{initially untouched}. \\

\textbf{Case 5.1. }$m=2$. \\

By Case 1 it follows that $v_{2}$ cannot be touched. Hence by the
choice of $m$ we must have $d\left(v_{2}\right)\geq3$. 
Suppose that Isolator claims the edge $v_{1}v_{2}$ on her first move. If Toucher
claims the edge $v_{0}v_{1}$ on her first move, we stop. Otherwise
Isolator claims the edge $v_{0}v_{1}$ on her second
move, and we stop after Toucher's second move. 

First consider the case when Isolator managed to claim both of these
edges, and consider $T_{1}$ obtained by deleting the vertices $v_{0}$
and $v_{1}$. Take $C_{1}=C'\cup\left\{ f_{1},f_{2}\right\} $, and
recall that $f_{1}$ and $f_{2}$ are the edges claimed by Toucher
on her two moves. It is easy to see that $T_{1}$ is a reduction of
$T$, and we have $d\left(n\right)=2$, $d\left(l\right)=1$ and $d\left(k\right)=-2$.
Note that $v_{2}$ is the only vertex whose degree is affected during
the process, and $v_{2}$ is initially unoccupied. Hence we have $d\left(s\right)\leq0$,
as the additional two moves given to Toucher can only decrease the
value of $d\left(s\right)$. Since $\left|D\right|=2$, it follows
that $D\left(T,T_{1}\right)=5\leq5\left(\left|D\right|-1\right)$,
as required. 

Now suppose that Toucher claimed the edge $v_{0}v_{1}$. Again consider
$T_{1}$ obtained by deleting the vertices $v_{0}$ and $v_{1}$,
but in this case we take $C_{1}=C'$. It is easy to see that $T_{1}$
is a reduction of $T$, and similarly we have $d\left(n\right)=2$,
$d\left(l\right)=1$, $d\left(k\right)=0$ and $d\left(s\right)=0$.
Indeed, in this case we have $d\left(s\right)=0$ as Toucher's only
additional move is claiming the edge $v_{0}v_{1}$. Since $\left|D\right|=1$
it follows that $D\left(T,T_{1}\right)=-1\leq5\left(\left|D\right|-1\right)$,
as required.\\

\textbf{Case 5.2. }$m=3$ and $v_{3}$ is initially touched. \\

Since $T$ has at least $6$ vertices, it follows that $v_{3}$ cannot
be a leaf, and hence it is occupied. Suppose that Isolator claims
the edge $v_{1}v_{2}$ on her first move and one of the edges in $\left\{ v_{0}v_{1},v_{2}v_{3}\right\} $
on her second move. If Toucher has claimed the other one of these
edges on one of her first two moves, the process stops after Toucher's
second move. Otherwise Isolator claims the other edge in $\left\{ v_{0}v_{1},v_{2}v_{3}\right\} $,
and the process stops after Toucher's third move. The rest of our
analysis splits into cases based on the number of neighbours of $v_{3}$.\\

\textbf{Case 5.2.1. }$d\left(v_{3}\right)=2$. \\

Let $v_{4}$ be chosen such that $d\left(v_{3}\right)=\left\{ v_{2},v_{4}\right\} $.
Since $v_{3}$ is occupied, it follows that $v_{3}v_{4}\in C$. We
also need to split into cases based on the number of neighbours of
$v_{4}$, and note that $v_{4}$ cannot be a leaf as $T$ has at least
$6$ vertices. \\

\textbf{Case 5.2.1.1. $d\left(v_{4}\right)=2$}.\\

Let $T_{1}$ be the tree obtained deleting the vertices $v_{0}$,
$v_{1}$, $v_{2}$ and $v_{3}$, and let $C_{1}$ be the set of those
edges in $C$ that are not deleted during the process. Since $v_{4}$
is a leaf in $T_{1}$ and occupied in $T$, it follows that $T_{1}$
is a reduction of $T$. 

First suppose that Isolator claimed all three edges in $\left\{ v_{0}v_{1},v_{1}v_{2},v_{2}v_{3}\right\} $.
Since $v_{4}$ is a leaf in $T_{1}$ and the edge $v_{3}v_{4}\in C$
is deleted during the process, it is easy to check that $d\left(n\right)=4$,
$d\left(l\right)=0$, $d\left(k\right)=-2$ and $d\left(s\right)\leq0$,
as the new edges claimed by Toucher cannot increase the value of $d\left(s\right)$.
Hence it follows that $D\left(T,T_{1}\right)\leq10=5\left(\left|D\right|-1\right)$,
which completes the proof of this case as $\left|T\right|>\left|T_{1}\right|$. 

Now suppose that Isolator claimed only two such edges. Hence one of
the edges in $\left\{ v_{0}v_{1},v_{1}v_{2},v_{2}v_{3}\right\} $
must be claimed by Toucher, and this edge is deleted together with
$v_{3}v_{4}\in C$. Hence it is easy to check that $d\left(n\right)=4$,
$d\left(l\right)=0$, $d\left(k\right)=0$ and $d\left(s\right)\leq0$,
and thus $D\left(T,T_{1}\right)\leq4<5\left(\left|D\right|-1\right)$
which completes the proof of this case. \\

\textbf{Case 5.2.1.2. }$d\left(v_{4}\right)\geq3$.\\

Let $N\left(v_{4}\right)\setminus\left\{ v_{3}\right\} =\left\{ u_{1},\dots,u_{a}\right\} $
where $a\geq2$. Let $T_{1}$ be the tree obtained by deleting the
vertices $v_{0}$, $v_{1}$ and $v_{2}$, and by replacing the edge
$v_{4}u_{1}$ with $v_{3}u_{1}$. Let $C_{1}$ be the set of all edges
in $C$ that are also edges in $T_{1}$, and if $v_{4}u_{1}\in C$
the edge $v_{3}u_{1}$ is also added to $C_{1}$. Hence it is easy
to see that $T_{1}$ is a reduction of $T$ by taking $f\left(v_{4}u_{1}\right)=v_{3}u_{1}$.
Note that the only vertices whose degrees are affected during the process
are $v_{3}$ and $v_{4}$, and it is easy to check that $d_{s}\left(v_{3}\right)=0$
and $d_{s}\left(v_{4}\right)=\left(a+1-2\right)-\left(a-2\right)=1$.
In particular, it follows that $d\left(s\right)\leq1$, and we certainly
also have $d\left(n\right)=3$. 

If Isolator claimed claimed all three edges, it follows that $d\left(l\right)=1$
and $d\left(k\right)=-3$. Hence we have $D\left(T,T_{1}\right)\leq10=5\left(\left|D\right|-1\right)$,
as required. If Isolator claimed only two such edges, it follows that
$d\left(l\right)=1$ and $d\left(k\right)=-1$. Hence we have $D\left(T,T_{1}\right)\leq4<5\left(\left|D\right|-1\right)$,
which completes the proof of this case. \\

\textbf{Case 5.2.2. }$d\left(v_{3}\right)\geq3$\@.\\

Let $T_{1}$ be the tree obtained by deleting the vertices $v_{0}$,
$v_{1}$ and $v_{2}$, and note that $T_{1}$ is a reduction of $T$.
Since $v_{3}$ is the only vertex whose degree is affected during the 
process and $d_{s}\left(v_{3}\right)=1$, it follows that $d\left(s\right)\leq1$.
We also have $d\left(n\right)=3$. 

If Isolator claimed all three edges, it follows that $d\left(l\right)=1$
and $d\left(k\right)=-3$. Hence we have $D\left(T,T_{1}\right)=10\leq5\left(\left|D\right|-1\right)$,
as required. If Isolator claimed only two such edges, it follows that
$d\left(l\right)=1$ and $d\left(k\right)=-1$. Hence we have $D\left(T,T_{1}\right)\leq4<5\left(\left|D\right|-1\right)$,
which completes the proof of this case. \\

These sub-cases cover the case when $m=3$ and $v_{3}$ is initially touched
completely. \\

\textbf{Case 5.3. }$m=3$ and $v_{3}$ is initially unoccupied. \\

Since $v_{3}$ is initially unoccupied and $m=3$, it follows that $d\left(v_{3}\right)\geq3$.
Again, suppose that Isolator claims the edge $v_{1}v_{2}$ on her
first move and one of the edges in $\left\{ v_{0}v_{1},v_{2}v_{3}\right\} $
on her second move. If Toucher has occupied the other one of these
edges on her first two moves, then the process stops after the second
move of Toucher. Otherwise Isolator claims the other one of these
edges on her third move and the process tops after the third move
of Toucher. 

First suppose that Isolator claimed all three edges, and let $T_{1}$
be the tree obtained by deleting the vertices $v_{0}$, $v_{1}$ and
$v_{2}$, and take $C_{1}=C$. Then $T_{1}$ is a reduction of $T$,
and the only vertex whose degree is affected during the process is
$v_{3}$. Since $v_{3}\not\in O'$ it follows that $d\left(s\right)\leq0$,
and it is also easy to check that $d\left(n\right)=3$, $d\left(l\right)=1$
and $d\left(k\right)=-3$. Hence we have $D\left(T,T_{1}\right)\leq9<5\left(\left|D\right|-1\right)$,
as required. 

If Isolator claimed only the edges $v_{1}v_{2}$ and $v_{2}v_{3}$,
we consider the same reduction $T_{1}$ as in the previous case, and
we take $C_{1}$ to be those edges in $C$ that are also edges in
$T_{1}$. Again, it is easy to see that $T_{1}$is indeed a reduction
of $T$ as the edge $v_{2}v_{3}$ is occupied by Isolator. It is also
easy to check that $d\left(n\right)=3$, $d\left(l\right)=1$, $d\left(k\right)=-1$
and $d\left(s\right)\leq0$, and hence we have $D\left(T,T_{1}\right)\leq3<5\left(\left|D\right|-1\right)$,
as required. 

Finally suppose that Isolator claimed only the edges $v_{0}v_{1}$
and $v_{1}v_{2}$, and hence Toucher has claimed the edge $v_{2}v_{3}$. 
Let $N\left(v_{3}\right)\setminus\left\{ v_{2}\right\} =\left\{ u_{1},\dots,u_{a}\right\} $
where $a\geq2$. Consider $T_{1}$ obtained by deleting the vertices
$v_{0}$ and $v_{1}$, and replacing the edge $v_{3}u_{1}$ with $v_{2}u_{1}$,
and set $C_{1}$ to be the all edges in $C$ that are also in
$T_{1}$, and if $v_{3}u_{1}\in C$ then $v_{2}u_{1}$ is also added
to $C_{1}$. Since both $v_{2}$ and $v_{3}$ are touched, it follows
that $T_{1}$ is a reduction of $T$ by taking $f_{E}\left(v_{3}u_{1}\right)=v_{2}u_{1}$. 

Note that $v_{2}$ and $v_{3}$ are the only vertices whose degrees
are affected during the process. Since both of them are initially
unoccupied, it follows that $d\left(s\right)\leq0$. It is easy to
check that $d\left(n\right)=2$, $d\left(l\right)=1$ and $d\left(k\right)=-2$.
Hence we have $D\left(T,T_{1}\right)\leq5=5\left(\left|D\right|-1\right)$,
which completes the proof of this case. \\

\textbf{Case 5.4. $m\geq4$}.\\

Suppose that Isolator claims the edge $v_{2}v_{3}$ on her first move.
Suppose that before a given move of Isolator the set of the edges
claimed by Isolator is of the form $\left\{ v_{i}v_{i+1},v_{i+1}v_{i+2},\dots,v_{j}v_{j+1}\right\} $
for some $i\leq2$ and $2\leq j\leq m-1$. If $j<m-1$ and if the
edge $v_{j+1}v_{j+2}$ is still available, Isolator claims this edge
on her move. Otherwise, if $i\geq1$ and the edge $v_{i-1}v_{i}$
is still available, Isolator claims this edge on her move. If neither
of these conditions is satisfied, the process stops. In particular, 
the process always stops after Toucher's move. 

Let $D=\left\{ v_{i}v_{i+1},\dots,v_{j}v_{j+1}\right\} $ be the set
of edges claimed by Isolator at the end of the process. In particular, we
have $\left|D\right|=j-i+1$ and hence the number of isolated vertices
is $j-i$, and the set of isolated vertices is $\left\{ v_{i+1},\dots,v_{j}\right\} $.
Note that we always have $i\leq2$, $2\leq j\leq m-1$ and $j-i\geq1$,
as Toucher cannot claim both $v_{1}v_{2}$ and $v_{3}v_{4}$ on her
first move. We now split into cases, mostly based on the value of
$j$ but sometimes also based on whether $v_{m}$ is touched or $d\left(v_{m}\right)\geq3$.
\\

\textbf{Case 5.4.1. }$j<m-2$.\\

Since $j\neq m-1$ at the end of the process, it follows that Toucher
has claimed the edge $v_{j+1}v_{j+2}$ on one of her moves. Let $T_{1}$
be the tree obtained by deleting the vertices $v_{0},\dots,v_{j+1}$,
and take $C_{1}$ to be the set of those edges in $C$ that are also
edges in $T_{1}$. Since $j<m-2$, it follows that $v_{j+2}$ is a
leaf in $T_{1}$. Hence $T_{1}$ is a reduction of $T$, and it is
easy to see that $d\left(n\right)=j+2$, $d\left(l\right)=0$ and
$d\left(s\right)\leq0$. 

First suppose that $i\in\left\{ 1,2\right\} $. Hence Toucher has
claimed at least one of the edges in $\left\{ v_{0}v_{1},v_{1}v_{2}\right\} $,
and also note that the edge $v_{j+1}v_{j+2}$ claimed by Toucher is
deleted. Since Toucher has claimed exactly $j-i+1$ edges outside
$C'$ as her new moves, it follows that $d\left(k\right)\geq\left(i-j-1\right)+2=1+i-j$.
Hence we have $D\left(T,T_{1}\right)\leq4j-3i-1$, and by using $i\leq2$
and $j-i\geq1$ it follows that $5\left(\left|D\right|-1\right)-D\left(T,T_{1}\right)\ge j-2i+1\geq0$,
as required.

Now suppose that $i=0$. In this case it follows that $d\left(k\right)=-j$,
as $v_{j+1}v_{j+2}$ is the only deleted edge claimed by Toucher.
Hence we have $D\left(T,T_{1}\right)\leq4j+2$, and since $j\geq2$
it follows that $D\left(T,T_{1}\right)\leq5j=5\left(\left|D\right|-1\right)$.
This completes the proof of this case. \\

\textbf{Case 5.4.2. }$j=m-2$ and $d\left(v_{m}\right)=2$.\\

Note that since $d\left(v_{m}\right)=2$, it follows that $v_{m}$
is initially touched by the definition of $m$. Again, since $j\neq m-1$
at the end of the process, it follows that Toucher has claimed the
edge $v_{m-1}v_{m}$ on one of her moves. Let $T_{1}$ be the tree
obtained by deleting the vertices $v_{0},\dots,v_{m-1}$, and take
$C_{1}$ to be the set of those edges in $C$ that are also edges
in $T_{1}$. Since $d\left(v_{m}\right)=2$, it follows that $v_{m}$
is a leaf in $T_{1}$. Hence $T_{1}$ is a reduction of $T$, and
it is easy to check that $d\left(n\right)=m$, $d\left(l\right)=0$
and $d\left(s\right)\leq0$. 

If $i\in\left\{ 1,2\right\} $, it follows that $d\left(k\right)\geq1+i-j=3+i-m$
by using the same argument as in Case 5.4.1. Hence we have $D\left(T,T_{1}\right)\leq4m-3i-9$,
and thus it follows that $5\left(\left|D\right|-1\right)-D\left(T,T_{1}\right)\geq m-2i-1$.
Since $j\geq i+1$ and $i\leq2$ it follows that $m=j+2\geq i+3\geq2i+1$,
as required. If $i=0$, it follows that $d\left(k\right)=-j=2-m$.
Hence we have $D\left(T,T_{1}\right)\leq4m-6$, and since $m\geq4$
it follows that $D\left(T,T_{1}\right)\leq5\left(m-2\right)=5\left(\left|D\right|-1\right)$,
which completes the proof of this case. \\

\textbf{Case 5.4.3. }$j=m-2$ and $d\left(v_{m}\right)\geq3$. \\

Since $j\neq m-1$, it again follows that Toucher must have claimed
the edge $v_{m-1}v_{m}$ on one of her moves. Let $N\left(v_{m}\right)\setminus\left\{ v_{m-1}\right\} =\left\{ u_{1},\dots,u_{a}\right\} $
with $a\geq2$. Consider $T_{1}$ obtained by removing the vertices
$v_{0},\dots,v_{m-2}$, and by replacing the edge $v_{m}u_{1}$ with
$v_{m-1}u_{1}$. Let $C_{1}$ to be the set of those edges in $C$
that are also edges in $T_{1}$, and if $v_{m}u_{1}\in C$ then $v_{m-1}u_{1}$
is also added to $C_{1}$. It is easy to see that $T_{1}$ is a reduction
of $T$ by taking $f\left(v_{m}u_{1}\right)=v_{m-1}u_{1}$ since both
$v_{m-1}$ and $v_{m}$ are touched. 

If $v_{m}$ is initially unoccupied, it is clear that $d\left(s\right)\leq0$.
Otherwise, we have $d_{s}\left(v_{m}\right)=\left(a+1-2\right)-\left(a-2\right)=1$
and $d_{s}\left(v_{m-1}\right)=0$. Hence it follows that $d\left(s\right)\leq1$
in either case. We also certainly have $d\left(n\right)=m-1$ and $d\left(l\right)=1$. 

If $i\in\left\{ 1,2\right\} $, it follows that Toucher has claimed
at least one of the edges $v_{0}v_{1}$ or $v_{1}v_{2}$, and hence
we have $d\left(k\right)\geq i-j=i+2-m$.
Thus it follows that $D\left(T,T_{1}\right)\leq4m-3i-9$, and hence
\[
5\left(\left|D\right|-1\right)-D\left(T,T_{1}\right)\geq5\left(m-2-i\right)-\left(4m-3i-9\right)=m-2i-1.
\]
By using $m=j+2\geq i+3$ and  $i\leq2$ it follows that $D\left(T,T_{1}\right)\leq5\left(\left|D\right|-1\right)$,
as required. 

If $i=0$, we have $d\left(k\right)=-j-1=1-m$.
Hence it follows that $D\left(T,T_{1}\right)\leq4m-6$. Since $m\geq4$,
we have $D\left(T,T_{1}\right)\leq5\left(m-2\right)=5\left(\left|D\right|-1\right)$,
which completes the proof of this case. \\

\textbf{Case 5.4.4. }$j=m-1$ and $d\left(v_{m}\right)\geq3$.
\\

Let $T_{1}$ be the tree obtained by deleting the vertices $v_{0},\dots,v_{m-1}$.
Since Isolator has occupied all the edges $v_{i}v_{i+1},\dots,v_{m-1}v_{m}$,
it follows that $T_{1}$ is a reduction of $T$ regardless whether
$v_{m}$ is touched or not. Since $v_{m}$ is the only vertex whose
degree is affected during the process and $d_{T_{1}}\left(v_{m}\right)=d_{T}\left(v_{m}\right)-1$,
it follows that $d\left(s\right)\leq1$. We also have $d\left(n\right)=m$ and $d\left(l\right)=1$. 

If $i\in\left\{ 1,2\right\} $, it follows that $d\left(k\right)\geq i+1-m$. 
Hence we have $D\left(T,T_{1}\right)\leq4m-3i-5$,
and thus 
\[
5\left(\left|D\right|-1\right)-D\left(T,T_{1}\right)\geq5\left(m-i-1\right)-\left(4m-3i-5\right)=m-2i.
\]
By using $m\geq4$ and $i\leq2$ it follows that $D\left(T,T_{1}\right)\leq5\left(\left|D\right|-1\right)$,
as required. 

If $i=0$, it follows that $d\left(k\right)=-m$.
Hence we have $D\left(T,T_{1}\right)\leq4m-2$, and since $m\geq4$
it follows that $D\left(T,T_{1}\right)\leq5\left(m-1\right)\leq5\left(\left|D\right|-1\right)$,
which completes the proof of this case. \\

\textbf{Case 5.4.5. }$j=m-1$ and $d\left(v_{m}\right)\leq2$. \\

Since $T$ is not a path, we must have $d\left(v_{m}\right)=2$. Since $d\left(v_{m}\right)=2$, the definition of $m$ implies that
$v_{m}$ is touched. 
Hence let $v_{m+1}$ be chosen so that $N\left(v_{m}\right)=\left\{ v_{m-1},v_{m+1}\right\} $.
Since $v_{m}$ is touched and $v_{m-1}v_{m}\not\in C$, it follows
that $v_{m}v_{m+1}\in C$. We split into sub-cases based on the degree
of $v_{m+1}$. First of all, note that $v_{m+1}$ cannot be a leaf
since $T$ is not a path. \\

\textbf{Case 5.4.5.1.} $d\left(v_{m+1}\right)=2$. \\

Let $T_{1}$ be the tree obtained by deleting the vertices $v_{0},\dots,v_{m}$.
Since $v_{m+1}$ is touched in $T$ and leaf in $T_{1}$, it is easy
to see that $T_{1}$ is a reduction of $T$, and we have $d\left(n\right)=m+1$, $d\left(l\right)=0$ 
and $d\left(s\right)=0$. 

If $i\in\left\{ 1,2\right\} $ it follows that $d\left(k\right)\geq-\left(j-i+1\right)+2=i-m+2$, as Toucher
has at least two edges that are deleted during the process, namely $v_{m}v_{m+1}$
and one of $v_{0}v_{1}$ or $v_{1}v_{2}$. Hence we have $D\left(T,T_{1}\right)\leq4m-3i-5$,
and thus 
\[
5\left(\left|D\right|-1\right)-D\left(T,T_{1}\right)\geq5\left(m-i-1\right)-\left(4m-3i-5\right)=m-2i.
\]
Again by using $m\geq4$ and $i\leq2$ it follows that $D\left(T,T_{1}\right)\leq5\left(\left|D\right|-1\right)$,
as required. 

If $i=0$ it follows that $d\left(k\right)\geq-\left(j+1\right)+1=1-m$,
and thus we have $D\left(T,T_{1}\right)\leq4m-2$. Since $m\geq4$,
it follows that $D\left(T,T_{1}\right)\leq5\left(m-1\right)=5\left(\left|D\right|-1\right)$,
which completes the proof of this case. \\

\textbf{Case 5.4.5.2. }$d\left(v_{m+1}\right)\geq3$.\\

Let $N\left(v_{m+1}\right)\setminus\left\{ v_{m}\right\} =\left\{ u_{1},\dots,u_{a}\right\} $
where $a\geq2$, and let $T_{1}$ be the tree obtained by deleting
the vertices $v_{0},\dots,v_{m-1}$ and by replacing the edge $v_{m+1}u_{1}$
with $v_{m}u_{1}$. Let $C_{1}$ be the set of those edges in $C$
that are also edges in $T_{1}$, and if $v_{m+1}u_{1}\in C$ then
$v_{m}u_{1}$ is also added to $C_{1}$. Then $T_{1}$ is a reduction
of $T$ by taking $f_{E}\left(v_{m+1}u_{1}\right)=v_{m}u_{1}$. It
is easy to see that $d\left(n\right)=m$ and $d\left(l\right)=1$. Note that $v_{m}$ and $v_{m+1}$
are the only vertices whose degrees are affected during the process.
Since $d_{s}\left(v_{m+1}\right)\leq1$ and $d_{s}\left(v_{m}\right)=0$,
it follows that $d\left(s\right)\leq1$. 

If $i\in\left\{ 1,2\right\} $, it follows that $d\left(k\right)\geq-\left(j-i+1\right)+1=i+1-m$. Hence we have
$D\left(T,T_{1}\right)\leq4m-3i-5$. Since $m\geq4\geq2i$, it follows
that 
\[
5\left(\left|D\right|-1\right)-D\left(T,T_{1}\right)\geq5\left(m-i-1\right)-\left(4m-3i-5\right)=m-2i\geq0,
\]
as required.

If $i=0$ it follows that $d\left(k\right)=-m$,
and hence we have $D\left(T,T_{1}\right)\leq4m-2$. Since $m\geq4$,
it follows that $D\left(T,T_{1}\right)\leq5\left(m-1\right)\leq5\left(\left|D\right|-1\right)$,
which completes the proof of this case and the proof of Case 5. $\hfill\square$\\

Suppose that $T$ does not contain any configurations described in
Cases 1-5, and let $v_{0},\dots,v_{m}$ be a maximal path of vertices
in $T$ for which $v_{i}$ is an unoccupied vertex of degree $2$
for all $1\leq i\leq m-1$, and for which we have $v_{i}\in N\left(v_{i-1}\right)$
for all $1\leq i\leq m$. Since $T$ does not contain any configurations
described in Cases 1-5, it follows that $v_{0}$ and $v_{m}$ are
also unoccupied, and the maximality assumption implies that we must
have $d\left(v_{0}\right)\geq3$ and $d\left(v_{m}\right)\geq3$.
In our final case we suppose that there exists a such a path with
$m\geq3$. \\

\textbf{Case 6. }There exists $m\geq3$ and unoccupied vertices $v_{0},\dots,v_{m}$
satisfying $v_{i}\in N\left(v_{i-1}\right)$ for all $1\leq i\leq m$,
$d\left(v_{i}\right)=2$ for all $1\leq i\leq m-1$, $d\left(v_{0}\right)\geq3$
and $d\left(v_{m}\right)\geq3$. \\

Suppose that Isolator claims the edge $v_{1}v_{2}$ on her first move.
Suppose that before a given move of Isolator the set of edges claimed
by Isolator is of the form $\left\{ v_{i}v_{i+1},v_{i+1}v_{i+2},\dots,v_{j}v_{j+1}\right\} $
for some $i\in\left\{ 0,1\right\} $ and $j\leq m-1$. If $j<m-1$
and if the edge $v_{j+1}v_{j+2}$ is still available, Isolator claims
this edge on her move. Otherwise, if $i=1$ and the edge $v_{0}v_{1}$
is still available, Isolator claims this edge on her move. If neither
of these conditions is satisfied, the process stops.  

Let $\left\{ v_{i}v_{i+1},\dots,v_{j}v_{j+1}\right\} $ be the set
of edges claimed by Isolator at the end of such process. Note that
again we have $\left|D\right|=j-i+1$, and we also have $i\in\left\{ 0,1\right\} $,
$1\leq j\le m-1$ and $j-i\geq1$ since Toucher cannot claim both
edges $v_{0}v_{1}$ and $v_{2}v_{3}$ on her first move. We again
split into several cases based on the values of $i$ and $j$. \\

\textbf{Case 6.1. }$i=0$ and $j=m-1$. \\

Let $S$ be the graph obtained by deleting the vertices $v_{1},\dots,v_{m-1}$.
It is easy to see that $S$ consists of two connected components
both of which are trees. Let $a$ and $b$ be leaves chosen from distinct
connected components, let $T_{1}$ be the tree obtained by adding
the edge $ab$ to the graph $S$  as demonstrated in Figure 4, and set $C_{1}=C\cup\left\{ ab\right\} $.
Note that the set of leaves in $T_{1}$ is exactly $L\setminus\left\{ a,b\right\} $
since $d_{T_{1}}\left(v_{0}\right)\geq3-1=2$ and $d_{T_{1}}\left(v_{m}\right)\geq3-1=2$.
Since $ab\in C_{1}$, it follows that $T_{1}$ is a reduction of $T$. 

\begin{figure}
\centering
\caption{Construction of $T_{1}$. Green edges are edges claimed by Toucher and they are deleted during the process and replaced with the red dotted edge.}
\begin{tikzpicture}

\node[main node] (n1) at (1,6){}; 
\node[every node] (n2) at (2,8){};
\node[text width=0.5cm] at (1.9, 8){\ldots};
\node[main node] (n3) at (3.5,8.5){}; 
\node[text width=0.5cm] at (3.65, 8.9){$a$};
\node[every node](n4) at (1,5){};
\node[text width=0.5cm] at (0.9,5){\ldots};
\node[main node](n6) at (2,5){};
\node[text width=0.5cm] at (2.15, 4.6){$v_0$};
\node[main node](n7) at (3,6){};
\node[text width=0.5cm] at (3.15, 6.4){$v_1$};
\node[main node](n8) at (4,5){};
\node[text width=0.5cm] at (4.15, 4.6){$v_2$};
\node[every node](n9) at (4.8,5){};
\node[text width=0.5cm] at (5.025,5){\ldots};
\node[every node](n18) at (5.2,5){};
\node[main node] (n10) at (6,5){};
\node[text width=0.5cm] at (6.15, 4.6){$v_{m-2}$};
\node[main node] (n11) at (7,6){};
\node[text width=0.5cm] at (7.15, 6.4){$v_{m-1}$};
\node[main node] (n12) at (8,5){};
\node[text width=0.5cm] at (8.15, 4.6){$v_m$};
\node[every node] (n14) at (9,5){};
\node[text width=0.5cm] at (9.2,5){\ldots};
\node[main node] (n15) at (9,6){};
\node[every node](n16) at (8,8){};
\node[text width=0.5cm] at (8.2, 8){\ldots};
\node[main node] (n17) at (6.5,8.5){};
\node[text width=0.5cm] at (6.65, 8.9){$b$};

\draw (n6) -- (n1);
\draw (n6) -- (n4);
\draw (n1) -- (n2);
\draw (n2) -- (n3);
\draw[style={draw=green}] (n6) -- (n7);
\draw[style={draw=green}] (n7) -- (n8);
\draw[style={draw=green}] (n8) -- (n9);
\draw[style={draw=green}] (n18) -- (n10);
\draw[style={draw=green}] (n10) -- (n11);
\draw[style={draw=green}] (n11) -- (n12);
\draw (n12) -- (n14);
\draw (n12) -- (n15);
\draw (n15) -- (n16);
\draw (n16) -- (n17);
\draw[style={draw=red},dotted] (n17) -- (n3);
\end{tikzpicture}
\end{figure}
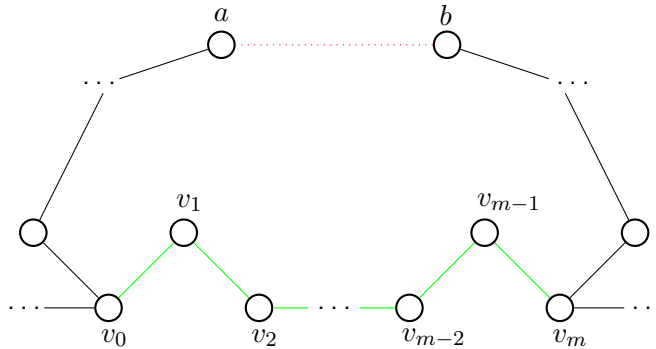

Since $v_{0}$ and $v_{m}$ are initially unoccupied, it is easy to
see that $d\left(s\right)\leq0$. Note that $d\left(k\right)=-m-1$,
as Toucher has claimed $m$ new edges and the edge $ab$ is assigned
to Toucher. Finally it is clear that we have $d\left(n\right)=m-1$
and $d\left(l\right)=2$. Since $m\geq3$, it follows that $D\left(T,T_{1}\right)\leq4m-4\leq5\left(m-1\right)=5\left(\left|D\right|-1\right)$,
which completes the proof of this case. \\

\textbf{Case 6.2. }$i=0$ and $j<m-1$. \\

Since $j<m-1$, it follows that Toucher has claimed the edge $v_{j+1}v_{j+2}$
on one of her moves. Let $S$ be the graph obtained by deleting the
vertices $v_{1},\dots,v_{j+1}$, and let $a$ be a leaf in the component
of $S$ containing $v_{0}$. Consider the tree $T_{1}$ obtained by
adding the edge $av_{j+2}$ to $S$, and define $C_{1}$ by setting
$C_{1}=\left(C\cup\left\{ av_{j+2}\right\} \right)\setminus\left\{ v_{j+1}v_{j+2}\right\} $.
Note that $T_{1}$ is a reduction of $T$ as both $a$ and $v_{j+2}$
are touched vertices in both $T$ and $T_{1}$. Finally note that
$d_{T_{1}}\left(v_{0}\right)\geq3-1=2$, and hence $v_{0}$ is not
a leaf in $T_{1}$. 

Note that the only vertices whose degrees are affected during the
process are $a$, $v_{0}$ and $v_{j+2}$. Note that $d_{s}\left(a\right)=0$,
and since $v_{0},\,v_{j+2}\not\in O'$ it follows that $d\left(s\right)\leq0$.
It is easy to check that we also have $d\left(n\right)=j+1$, $d\left(l\right)=1$
and $d\left(k\right)=-\left(j+1\right)$. Thus we have $D\left(T,T_{1}\right)\leq4j+1$,
and since $j\geq1$ it follows that $D\left(T,T_{1}\right)\leq5j=5\left(\left|D\right|-1\right)$,
which completes the proof of this case. \\

\textbf{Case 6.3. }$i=1$ and $j=m-1$. \\

Note that the case $\left(i,j\right)=\left(1,m-1\right)$ is equivalent
to the case $\left(i,j\right)=\left(0,m-2\right)$, which is covered
in the previous case. \\

\textbf{Case 6.4. }$i=1$ and $j<m-1$.\\

Since $i=1$ and $j<m-1$, it follows that Toucher has claimed the
edges $v_{0}v_{1}$ and $v_{j+1}v_{j+2}$. Let $T_{1}$ be the tree
obtained by deleting the vertices $v_{1},\dots,v_{j+1}$ and by adding
the edge $v_{0}v_{j+2}$, and set $C_{1}=\left(C\cup\left\{ v_{0}v_{j+2}\right\} \right)\setminus\left\{ v_{0}v_{1},v_{j+1}v_{j+2}\right\} $.
Note that $T_{1}$ is a reduction of $T$ as both $v_{0}$ and $v_{j+2}$
are touched before and after the reduction. 

Note that the degree of any vertex that is not deleted is not affected
during the process. Since $v_{0}\in O'\setminus O_{1}$ and $d_{T_{1}}\left(v_{0}\right)=d_{T}\left(v_{0}\right)\geq3$,
it follows that $d_{s}\left(v_{0}\right)\leq-1$. Hence we must have
$d\left(s\right)\leq-1$. Since the edges $v_{0}v_{1}$ and $v_{j+1}v_{j+2}$
claimed by Toucher are deleted in the process and the edge $v_{0}v_{j+2}$
is given to Toucher, it follows that $d\left(k\right)=1-j$, and it
is easy to see that $d\left(n\right)=j+1$ and $d\left(l\right)=0$.
Hence we have $D\left(T,T_{1}\right)\leq4j-3$. Since 
$j\geq i+1\geq2$, it follows that $D\left(T,T_{1}\right)\leq5\left(j-1\right)=5\left(\left|D\right|-1\right)$,
which completes the proof of Case 6. $\hfill\square$\\

Let $C=\left\{ e_{1},\dots,e_{k}\right\} $ be the set of edges claimed
by Toucher at the start of the game. Our aim is to prove that if $T$
together with this particular collection $C$ does not contain any
of the configurations described in Cases 1-6, then it follows that
$S\left(T\right)\leq0$. For each $1\leq i\leq k$ let $e_{i}=a_{i}b_{i}$ 
and let $d_{i}=d_{T}\left(a_{i}\right)+d_{T}\left(b_{i}\right)-2$. 

We say that a graph $T$ is a \textit{forest} if every connected component
of $T$ is a tree. We define a sequence of forests $T_{0},\dots,T_{k}$
and collections of edges $C_{0},\dots,C_{k}$ as follows. First of
all, we set $T_{0}=T$ and $C_{0}=C=\left\{ e_{1},\dots,e_{k}\right\} $,
and at every stage we will have $C_{i}=\left\{ e_{i+1},\dots,e_{k}\right\} $. 

Given $T_{i}$ and $C_{i}$, let $X$ be the connected component of
$T_{i}$ containing the edge $e_{i}$, and note that $X$ is a tree
since $T_{i}$ is a forest. Let $Y$ be the forest consisting of $d_{i}$
trees obtained by removing the vertices $a_{i}$ and $b_{i}$ and
the edge $a_{i}b_{i}$, and by adding one new vertex to each connected
component $S$ of $Y$ joined by an edge to the vertex of $S$
that was neighbour of $a_{i}$ or $b_{i}$. Note that such a vertex
always exists in each connected component, and such vertex is also
unique since $X$ is a tree. Finally we set $T_{i+1}$ to be the union
of $Y$ and all the components of $T_{i}$ apart from $X$. One stage
of the process is illustrated in Figure 5. 

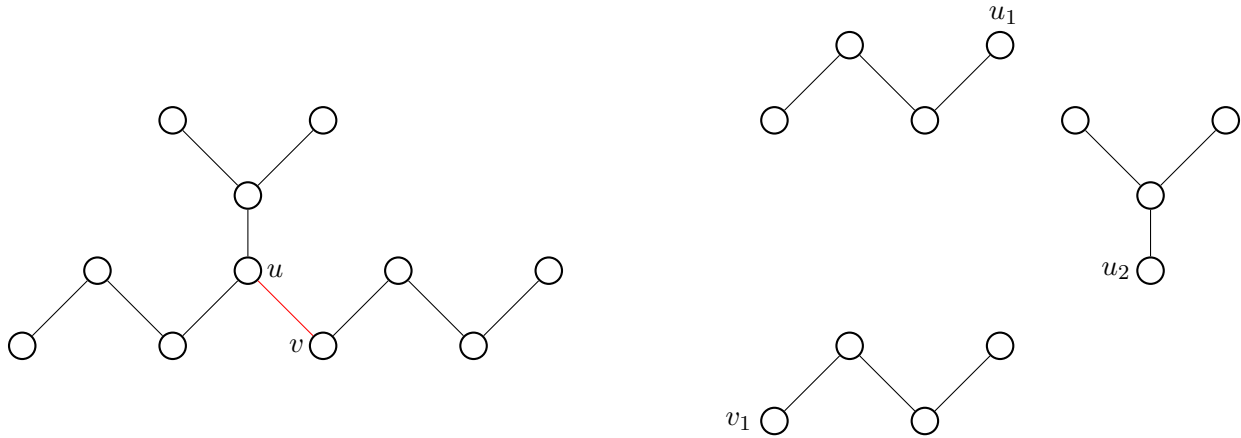
\begin{figure}
\centering
\caption{Illustration of one stage of the process. The edge $uv\in C_{i}$
is deleted, and since $d_{i}=3$ the connected component on left
splits into $3$ trees on right. }
\begin{tikzpicture}
\node[main node] (n1) at (1,5){};
\node[main node] (n2) at (2,6){};
\node[main node] (n3) at (3,5){};
\node[main node] (n4) at (4,6){};
\node[text width=0.5cm] at (4.5,6){$u$};
\node[main node] (n5) at (5,5){};
\node[text width=0.5cm] at (4.8,5){$v$};
\node[main node] (n6) at (6,6){};
\node[main node] (n7) at (7,5){};
\node[main node] (n8) at (8,6){};
\node[main node] (n9) at (4,7){};
\node[main node] (n10) at (3,8){};
\node[main node] (n11) at (5,8){};

\draw (n1) -- (n2);
\draw (n2) -- (n3);
\draw (n3) -- (n4);
\draw (n4) -- (n9);
\draw (n5) -- (n6);
\draw (n6) -- (n7);
\draw (n7) -- (n8);
\draw[style={draw=red}] (n4) -- (n5);
\draw (n9) -- (n10);
\draw (n9) -- (n11);

\node[main node] (m1) at (11,8){};
\node[main node] (m2) at (12,9){};
\node[main node] (m3) at (13,8){};
\node[text width=0.5cm] at (14.1,9.4){$u_1$};
\node[main node] (m4) at (14,9){};
\node[text width=0.5cm] at (15.6,6){$u_2$};
\node[main node] (m5) at (16,6){};
\node[main node] (m6) at (16,7){};
\node[main node] (m7) at (15,8){};
\node[main node] (m8) at (17,8){};
\node[main node] (m9) at (11,4){};
\node[text width=0.5cm] at (10.6,4){$v_1$};
\node[main node](m10) at (12,5){};
\node[main node](m11) at (13,4){};
\node[main node](m12) at (14,5){};
\draw (m1) -- (m2);
\draw (m2) -- (m3);
\draw (m3) -- (m4);
\draw (m5) -- (m6);
\draw (m6) -- (m7);
\draw (m6) -- (m8);
\draw (m9) -- (m10);
\draw (m10) -- (m11);
\draw (m11) -- (m12);
\end{tikzpicture}
\end{figure}

Note that by Claims 2, 3 and 4 it follows that all $a_{1},\dots,a_{k},b_{1},\dots,b_{k}$
are distinct vertices, none of them is a leaf in any $T_{i}$ and any two
such vertices are neighbours if and only if they are $a_{j}$ and
$b_{j}$ for some $j$. Also by Claim 1 it follows that every connected
component in $T_{k}$ has at least $4$ vertices. 

Note that during the $i^{th}$ step of the process, the number of
connected components increases by $d_{i}-1$, as one connected component
splits into $d_{i}$ connected components. Hence the number of connected
components in $T_{k}$ is 
\begin{equation}
D=1+\sum_{i=1}^{k}\left(d_{i}-1\right)=1-k+\sum_{i=1}^{k}d_{i}.\label{eq:M1}
\end{equation}

Let $n_{1},\dots,n_{D}$ be the number of vertices in each connected
component and let $l_{1},\dots,l_{D}$ be the number of leaves in
each connected component. Note that on the $i^{th}$ stage the number
of vertices increases by $d_{i}-2$, as we delete the vertices $a_{i}$
and $b_{i}$ and add $d_{i}$ new vertices that are leaves. Hence
we have 
\begin{equation}
\sum_{i=1}^{D}n_{i}=n+\sum_{i=1}^{k}\left(d_{i}-2\right)=n-2k+\sum_{i=1}^{k}d_{i}.\label{eq:M2}
\end{equation}
Since none of the vertices $a_{1},\dots,a_{k},b_{1},\dots,b_{k}$
is a leaf at any stage of the process before they are deleted, it
follows that the number of leaves increases by $d_{i}$ on the $i^{th}$
stage. Hence we have 
\begin{equation}
\sum_{i=1}^{D}l_{i}=l+\sum_{i=1}^{k}d_{i}.\label{eq:M3}
\end{equation}

Let $S$ be a connected component in $T_{k}$. Note that if $S$ contains
a vertex of degree $2$ whose neighbour is a leaf, then we can backtrack
the process to find a vertex of degree $2$ in $T$ whose neighbour
is a touched vertex, which contradicts Case 5. Hence we may assume
that no vertex of degree $2$ in $S$ has a leaf as a neighbour. 

If $S$ contains two vertices of degree $2$ that are neighbours,
it follows that there exists a path of vertices $v_{0},\dots,v_{t}$
in $S$ for some $t\geq3$ with $d\left(v_{0}\right)\geq3$, $d\left(v_{t+1}\right)\geq3$
and $d\left(v_{i}\right)=2$ for all $1\leq i\le t$. Since none of
these is a leaf in $S$, it follows that these vertices also formed
a path satisfying the same condition in $T$, and all of these vertices
are unoccupied in $T$. This contradicts Case 6. 

Hence in every connected component there is no vertex of degree $2$
whose neighbour is a leaf or another vertex of degree $2$. Since
each connected component is a tree with at least $4$ vertices, Lemma
\ref{lem:Lemma5} implies that $3l_{i}\geq n_{i}+5$. Adding these
inequalities for all $i\in\left\{ 1,\dots,D\right\} $, and using
(\ref{eq:M1}), (\ref{eq:M2}) and (\ref{eq:M3}) it follows that
\[
3\left(l+\sum_{i=1}^{k}d_{i}\right)\geq n-2k+\sum_{i=1}^{k}d_{i}+5-5k+5\sum_{i=1}^{k}d_{i}.
\]
This can be rearranged to 
\begin{equation}
3l+3k\geq n+5+3\sum_{i=1}^{k}d_{i}-4k.\label{eq:M4}
\end{equation}

Note that $O\left(T\right)=\left\{ a_{1},\dots,a_{k},b_{1},\dots,b_{k}\right\} $,
and hence it follows that 
\[
\sum_{v\in O\left(T\right)}\left(d\left(v\right)-2\right)=\sum_{i=1}^{k}\left(d\left(a_{i}\right)+d\left(b_{i}\right)-4\right)=\sum_{i=1}^{k}\left(d_{i}-2\right)=-2k+\sum_{i=1}^{k}d_{i}.
\]
Hence (\ref{eq:M4}) can be written as 
\begin{equation}
3l+3k\ge n+5+\sum_{v\in O\left(T\right)}\left(d\left(v\right)-2\right)+2\sum_{i=1}^{k}d_{i}-2k.\label{eq:M4.5}
\end{equation}

Since none of $a_{i}$ or $b_{i}$ is a leaf, it follows that $d_{i}=d\left(a_{i}\right)+d\left(b_{i}\right)-2\geq2$.
Hence we have $2\sum_{i=1}^{k}d_{i}-2k\geq2k\geq0$. In particular,
(\ref{eq:M4.5}) implies that 
\[
n+7-3k-3l+\sum_{v\in O\left(T\right)}\left(d\left(v\right)-2\right)\leq2,
\]
and thus we must have $S\left(T\right)\leq\left\lfloor \frac{2}{5}\right\rfloor =0$.
Hence the claim follows trivially as we always have $\alpha\left(T,C,L\right)\geq0$.
Since we always have $\sum_{v\in O\left(T\right)}\left(d\left(v\right)-2\right)\geq0$,
the second part follows immediately. This completes the proof of Lemma
\ref{lem:Lemma 4}. 
\end{proof}

We are now ready to prove Theorem \ref{thm:1}. 

\begin{proof} [Proof of Theorem \ref{thm:1}] Let $T$
be a tree with $n\geq3$ vertices. Suppose that during the first phase of
the game Isolator follows the strategy specified in Lemma \ref{lem:Lemma 3},
and let $r$ be the number of edges claimed by her during the first
phase of the game. Let $T'$, $C'$ and $X'=L'$ be given as in Lemma
\ref{lem:Lemma 3}. Since $\left|I\right|=r$, it follows that $\alpha\left(T\right)\geq r+\alpha\left(T',C',L'\right)$.
Since the second phase is equivalent to the delayed game $F\left(T',C',L'\right)$,
Lemma \ref{lem:Lemma5} implies that 
\[
\alpha\left(T',C',L'\right)\geq\left\lfloor \frac{\left|T'\right|-3\left|C'\right|-3\left|L'\right|+7}{5}\right\rfloor .
\]
Since Lemma \ref{lem:Lemma 3} guarantees that 
\[
\left|T'\right|-3\left|C'\right|-3\left|L'\right|\geq n-5r-4,
\]
it follows that 
\[
\alpha\left(T\right)\geq r+\left\lfloor \frac{n-5r-4+7}{5}\right\rfloor =\left\lfloor \frac{n+3}{5}\right\rfloor ,
\]
which completes the proof of Theorem (\ref{thm:1}). 
\end{proof}

There are many questions that are open concerning the value of $u\left(G\right)$
for general $G$. Dowden, Kang, Mikala\v{c}ki and Stojakovi\'{c} \cite{key-7}
gave bounds for $u(G)$ that depend on the degree sequence of the
graph $G$. In particular, they concluded that if the minimum degree
of $G$ is at least $4$ then $u\left(G\right)=0$. They also proved
that there exists a $3$-regular graph with $u\left(G\right)>0$,
and they proved that for all $3$-regular graphs we have $u\left(G\right)\leq\frac{n}{8}$.
It would be interesting to know what is the largest possible proportion
of untouched vertices in a connected $3$-regular graph.

\end{document}